\pdfoutput=1

\documentclass[
final,
disable,
a4paper,
12pt,
]{article}

\usepackage{artmacs}

\usepackage{mathdots}    %

\usepackage{tikz}
\usetikzlibrary{decorations.pathreplacing, decorations.text}
\usetikzlibrary{arrows}
\usetikzlibrary{external}
\makeatletter
\def\namedlabel#1#2{\begingroup
    \def\@currentlabel{\protect{\textnormal{#2}}}%
    \label{#1}\endgroup
}

\newcommand{\Cfrob}[1]{C_{#1}^{(F)}}
\newcommand{\Csimp}[1]{C_{#1}^{(S)}}
\newcommand{\Cmult}[1]{C_{#1}^{(M)}}

\newcommand{\SimpConst}[1]{S({#1})} %
\newcommand{\MultConst}[1]{M({#1})} %

\newcommand{\dif}[1]{{{#1}'}}

\newcommand{\addpol}{A}
\newcommand{\subadd}{S}
\newcommand{\Oring}{\mathcal{O}}

\newcommand{\prip}{\mathfrak{p}}
\newcommand{\priP}{\mathfrak{q}}
\newcommand{\priq}{\mathfrak{p^*}}
\newcommand{\et}{e^*}

\newcommand{\ff}{F}
\newcommand{\kk}{K}

\newcommand{\ii}{i}
\newcommand{\jj}{j}
\newcommand{\alt}[1]{#1_{0}}
\newcommand{\qepol}{H}
\newcommand{\divfct}{\tau} %
\newcommand{\MM}{\mathsf{M}}

\DeclareMathOperator{\mult}{mult}

\newcommand{\fnull}{f_{0}}
\newcommand{\feins}{f_{1}}
\newcommand{\fzwei}{f_{2}}
\newcommand{\fdrei}{f_{3}}

\newcommand{\Van}[2]{#1^{-1}(#2)}   %

\makeatletter
\def\blfootnote{\gdef\@thefnmark{}\@footnotetext}
\makeatother

\begin{document}

\title{Compositions and collisions at degree $p^{2}$}
\pdftitle{Compositions and collisions at degree p^2}
\author{
Raoul Blankertz, Joachim von~zur~Gathen \& Konstantin Ziegler\\
B-IT, Universit\"at Bonn\\
D-53113 Bonn, Germany\\
\email{blankertz@uni-bonn.de, {gathen,zieglerk}@bit.uni-bonn.de}\\
\url{http://cosec.bit.uni-bonn.de/}
}
\pdfauthor{Raoul Blankertz, Joachim von zur Gathen, Konstantin Ziegler}
\maketitle

\begin{abstract}
A univariate polynomial $f$ over a field is decomposable if $f= g
\circ h= g(h)$ for nonlinear polynomials $g$ and $h$.  In order to
count the decomposables, one wants to know, under a suitable normalization, the number of equal-degree collisions of the form $f = g \circ h = g^* \circ h^*$ with $(g,h) \neq (g^{*}, h^{*})$ and $\deg g = \deg g^*$.  Such collisions only occur in the wild case, where the field characteristic $p$ divides $\deg f$.  Reasonable bounds on the number of decomposables over a finite field are known, but they are less sharp in the wild case, in particular for degree $p^2$.

We provide a classification of all polynomials of degree $p^2$ with a
collision.  It yields the exact number of decomposable polynomials of
degree $p^{2}$ over a finite field of characteristic $p$.  We also
present an efficient algorithm that determines whether a given polynomial of degree $p^{2}$ has a collision or not.
\end{abstract}

\begin{keywords}
computer algebra, finite fields, wild polynomial decomposition, equal-degree
collisions, ramification theory of function fields, counting special polynomials
\end{keywords}

\section{Introduction}

The \emph{composition} of two polynomials $g,h \in F[x]$ over a field
$F$ is denoted as $f= g \circ h= g(h)$, and then $(g,h)$ is a
\emph{decomposition} of $f$, and $f$ is \emph{decomposable} if $g$ and
$h$ have degree at least $2$.  In the 1920s, Ritt, Fatou, and Julia
studied structural properties of these decompositions over
$\mathbb{C}$, using analytic methods. Particularly important are two
theorems by Ritt on the uniqueness, in a suitable sense, of
decompositions, the first one for (many) indecomposable components and
the second one for two components, as above.  \cite{eng41} and
\cite{lev42} proved them over arbitrary fields of characteristic zero
using algebraic methods.

The theory was extended to arbitrary characteristic by \cite{frimac69}, \cite{dorwha74}, \cite{sch82c,
  sch00c}, \cite{zan93}, and others. Its use in a cryptographic
context was suggested by \cite{cad85}. In computer algebra, the method
of \cite{barzip85} requires exponential time.  A fundamental
dichotomy is between the \emph{tame case}, where the characteristic
$p$ does not divide $\deg g$, and the \emph{wild case}, where $p$
divides $\deg g$, see \cite{gat90d,gat90c}.  A breakthrough result of \cite{kozlan89} was their
polynomial-time algorithm to compute tame decompositions. In the wild case, considerably less is known, both
mathematically and computationally. \cite{zip91} suggests that the
block decompositions of \cite{lanmil85} for determining subfields of
algebraic number fields can be applied to decomposing rational
functions even in the wild case.  A version of Zippel's algorithm in
\cite{bla13} computes in polynomial time all decompositions of a polynomial that are
minimal in a certain sense.  \cite{avazan03} study ambiguities in the
decomposition of rational functions over $\mathbb{C}$.  A set of distinct decompositions of $f$
is called a \emph{collision}. The number of decomposable polynomials of degree $n$ is thus
the number of all pairs $(g,h)$ with $\deg g \cdot \deg h = n$ reduced
by the ambiguities introduced by collisions.  In this paper, we study
only \emph{equal-degree} collisions of $f = g \circ h = g^{*} \circ
h^{*}$, where $\deg g = \deg g^{*}$ and thus $\deg h = \deg h^{*}$.

The task of counting compositions over a finite field of
characteristic $p$ was first considered in
\cite{gie88b}.
\Citet{gat09b} presents general approximations to the
number of decomposable polynomials. These come with satisfactory
(rapidly decreasing) relative error bounds except when $p$ divides $n
= \deg f$ exactly twice.
The main result (\autoref{cor:main}) of the present work determines exactly the number of decomposable polynomials in one of these difficult cases, namely when $n = p^{2}$ and hence $\deg g =\deg h = p$.

This is shown in three steps. First, we exhibit some classes of collisions in \autoref{sec:explicit-construction}. Their properties are easy to check.
In the second step we show that these are all possibilities
(\autoref{thm:normal}).   In Section~\ref{sec:algebra} we use
ramification theory of function fields to study the root
multiplicities in collisions, and in \autoref{sec:classification}
classify all collisions at degree $p^{2}$.
In the third step we count the resulting possibilities (\autoref{sec:Counting}).

Our contribution is fourfold:
\begin{itemize}
\item We provide explicit constructions for collisions at degree
  $r^{2}$, where $r$ is a power of the characteristic $p >
  0$ (\autoref{thm:nonadd}, \autoref{thm:constmulti}).
\item We provide a classification of all collisions at degree $p^{2}$,
  linking every collision to a unique explicit construction (\autoref{thm:normal}).
\item We use these two results to obtain an exact formula for the
  number of decomposable polynomials at degree $p^{2}$ (\autoref{cor:main}).
\item The classification yields an efficient algorithm to test whether
  a given polynomial of degree $p^{2}$ has a collision or not (\autoref{algo:coll-det}).
\end{itemize}

An Extended Abstract of this paper appeared as \citeauthor*{blagat13} (2012).
Notice: this is the authors' version of a work that was accepted for
publication in \emph{Journal of Symbolic Computation}. Changes
resulting from the publishing process, such as peer review, editing,
corrections, structural formatting, and other quality control
mechanisms may not be reflected in this document. Changes have been
made to this work since it was submitted for publication. A definitive
version was subsequently published as \cite{blagat13}.

\section{Definitions and examples}
\label{sec:prelim}
We consider
a field $F$ of positive characteristic $p > 0$.
Composition of $g$ and $h$ with linear polynomials introduces
inessential ambiguities in decompositions $f=g\circ h$.  To avoid them, we normalize $f$, $g$, and
$h$ to be \emph{monic original}, that is with leading coefficient 1
and constant coefficient 0 (so that the graph of $f$ passes through
the origin); see
\citet{gat09b}.

For a nonnegative integer $k$, an (equal-degree) \emph{$k$-collision} at degree $n$ is a set of
  $k$ distinct pairs $(g,h)$ of monic original polynomials in $\ff[x]$
  of degree at least $2$, all with the same
  composition $f = g \circ h$ of degree $n$ and $\deg g$ the same for
  all $(g,h)$.  A $k$-collision is called \emph{maximal} if it is not
  contained in a $(k+1)$-collision.  We also say that $f$ has a (maximal)
  $k$-collision.  Furthermore, $g$ is a \emph{left component} and $h$ a \emph{right component} of $f$.
For $n\geq 1$, we define
  \begin{align}
    P_{n}(\ff) & = \{ f \in \ff[x] \colon \text{$f$ is monic original of degree $n$}\}, \label{eq:45} \\
    D_{n}(\ff) & = \{ f \in P_{n}(\ff) \colon \text{$f$ is decomposable} \}, \label{eq:46} \\
    C_{n,k}(\ff) & = \{ f \in P_{n}(\ff) \colon \text{$f$ has a maximal $k$-collision} \}. \label{eq:47}
  \end{align}
Thus $\# P_n(\Fq) = q^{n-1}$. We sometimes leave out $\ff$ from the notation when it is clear from the context.

Let $f \in P_n$ have a $k$-collision $C$, $f' \neq 0$, and $m$ be a
divisor of $n$. If all right components in $C$ are of degree $m$ and
indecomposable, then $k \leq (n-1)/(m - 1)$; see \citet[Corollary
3.27]{bla11}. For $n = p^2$, both components are
of degree $p$ and thus indecomposable and we find $k \leq p+1$; see also \citet*[Proposition 6.5~(iv)]{gatgie10a}.
For counting all decomposable polynomials of degree $p^2$ over $\Fq$,
it is sufficient to count the sets $C_{p^2, k}$ of polynomials with maximal $k$-collision for $k \geq 2$, since
\begin{equation}
\label{eq:missing1}
\# D_{p^2} = q^{2p-2} - \sum_{k \geq 2} (k-1) \cdot \# C_{p^2,k}.
\end{equation}

\begin{lemma}
\label{lem:gunique}
In a decomposition $(g,h)$, $g$ is uniquely determined by $g\circ h$ and $h$.
\end{lemma}
\begin{proof}
Let $f=g\circ h$. Consider the $F$-algebra homomorphism $\varphi \colon \ff[x] \rightarrow \ff[x]$ with $x \mapsto h$. Its kernel is trivial, since $h$ is nonconstant, and thus $\varphi$ is injective. Hence there is exactly one $u \in \ff[x]$ such that $\varphi (u) = f$, namely $u = g$.
\end{proof}

Furthermore, $g$ is easy to compute from  $g\circ h$ and $h$ by the generalized Taylor expansion; see \citet[Section 2]{gat90c}.
The following is a simple example of a collision.
\begin{example}
\label{exa:frob}
Let $r = p^e$. For $h \in P_r(\ff)$, we have
\begin{equation}
\label{eq:17}
x^{r} \circ h = \varphi_{r} (h) \circ x^{r},
\end{equation}
where $\varphi_{r}$ is the $e$th power of the \emph{Frobenius
  endomorphism} on $\ff$, extended to polynomials coefficientwise. If
$h \neq x^r$, then $\{ (x^{r} ,h), ( \varphi_{r} (h) , x^{r}) \}$ is a
2-collision and we call it a \emph{Frobenius collision}.
\end{example}

In the case $r = p$, we have the following description.
\begin{lemma}
\label{lem:cor:frob}
\begin{ronumerate}
\item\label{lem:frob} Assume that $f \in P_{p^2}(F)$ has a 2-collision. Then it is a Frobenius collision if and only if $\dif{f} = 0$.
\item\label{cor:frob} A Frobenius collision of degree $p^2$ is a maximal 2-collision.
\end{ronumerate}
\end{lemma}
\begin{proof}
\ref{lem:frob} If $f$ is a Frobenius collision, then $f'=0$ by
definition.  Conversely, let $f \in P_{p^2}(\ff)$ with $\dif{f} = 0$. Then $f \in \ff[x^p]$ and
thus $f = g \circ x^p$ for some monic original polynomial $g$.  Let $f
= g^* \circ h^*$  be another decomposition of $f$. By
\autoref{lem:gunique}, $f$ and $h^{*}$ determine $g^{*}$ uniquely,
hence $h^* \neq x^p$ and $\dif{h^*} \neq 0$. Thus from $\dif{f} =
\dif{g^*} (h^*) \cdot \dif{h^*} = 0$ follows $\dif{g^*} = 0$ and hence
$g^* = x^p$. Furthermore, $f = x^p \circ h^* = \varphi_{p} (h^*) \circ x^{p}$ by
\eqref{eq:17}, $g = \varphi_{p}(h^{*})$ by
the uniqueness in \autoref{lem:gunique}, and $f$ is a Frobenius collision.

\ref{cor:frob} Let $f = x^{p} \circ h = \varphi_{p} (h) \circ x^{p}$, with $h \neq
x^p$, be a Frobenius collision, and $(g^*, h^*)$ a decomposition of $f$. Then $0 = \dif{f} =
\dif{g^*} (h^*) \cdot \dif{h^*}$ and thus $\dif{g^*} = 0$ or
$\dif{h^*} = 0$. If $\dif{h^*} = 0$, then $h^* = x^p$ and thus $g^* =
\varphi_{p} (h)$, by \autoref{lem:gunique}. If $\dif{g^*} = 0$, then
$g^* = x^p$ and $f =  \varphi_p (h^*) \circ x^p$ as in \ref{lem:frob}. Thus $\varphi_p (h^*) = \varphi_{p} (h)$ by the
uniqueness in \autoref{lem:gunique}, which implies $h = h^*$.
\end{proof}

If $\ff$ is perfect---in particular if $F$ is finite or algebraically closed---then the Frobenius endomorphism
$\varphi_p$ is an automorphism on $\ff$. Thus for $f \in
P_{p^2}(\Fq)$, $\dif{f} = 0$ implies that $f$ is either a
Frobenius collision or $x^{p^2}$.

For $f \in P_{n} (\ff)$ and $w \in \ff$, the \emph{original shift} of $f$ by $w$ is
\begin{equation}
  \label{eq:14}
  f^{[w]} = (x-f(w)) \circ f \circ (x+w) \in P_{n}(\ff).
\end{equation}
We also simply speak of a \emph{shift}. Original shifting defines a group action of the additive group of $\ff$ on $P_{n}(\ff)$. Indeed, we have for $w, w' \in \ff$
\begin{align*}
  (f^{[w]})^{[w']} & = (x-f^{[w]}(w')) \circ f^{[w]} \circ (x+w') \\
& = (x - (f(w'+w)-f(w))) \circ (x -
  f(w)) \circ f \circ (x+w) \circ (x+w') \\
& =(x - f(w'+w)) \circ f \circ (x+w'+w)  = f^{[w'+w]}.
\end{align*}
Furthermore, for the derivative we have $(f^{[w]})' = f' \circ (x+w)$.  Shifting respects decompositions in the sense that
for each decomposition $(g,h)$ of $f$ we have a decomposition $(g^{[h(w)]}, h^{[w]})$ of $f^{[w]}$, and vice versa.
We denote $(g^{[h(w)]}, h^{[w]})$ as $(g, h)^{[w]}$.

\section{Explicit collisions at degree {$r^{2}$}}
\label{sec:explicit-construction}

This section presents two classes of explicit collisions at degree $r^{2}$, where $r$
is a power of the characteristic $p>0$ of the field $F$.
The collisions of \autoref{thm:nonadd} consist of additive and subadditive polynomials.
A polynomial $\addpol$ of degree $r^{\kappa}$
is \emph{$r$-additive} if it is of the form $\addpol = \sum_{0 \leq i
  \leq \kappa} a_i x^{r^i}$ with all $a_i \in \ff$. We call a
polynomial \emph{additive} if it is $p$-additive. A polynomial is
additive if and only if it acts additively on an algebraic closure
$\overline{F}$ of $F$, that is $\addpol(a + b) = \addpol(a)+
\addpol(b)$ for all $a$, $b \in \overline{F}$; see \citet[Corollary
1.1.6]{gos96}. The composition of additive polynomials is additive,
see for instance Proposition 1.1.2 of the cited book.  The
decomposition structure of additive polynomials was first studied by
\cite{ore33b}.  \citet[Theorem~4]{dorwha74} show that all components
of an additive polynomial are additive.  \cite{gie88b} gives lower bounds
on the number of decompositions and algorithms to determine them.

For a divisor $m$ of $r -1$, the \emph{$(r,m)$-subadditive} polynomial
associated with the $r$-additive polynomial $\addpol$ is $\subadd =
x(\sum_{0 \leq i \leq \kappa} a_i x^{(r^i - 1)/m})^m$ of degree
$r^{\kappa}$.  Then $\addpol$ and $\subadd$ are related as $x^m \circ
\addpol = \subadd \circ x^m$ and fall into the First Case of Ritt's
Second Theorem.
\citet{dic97} notes a special case of subadditive polynomials, and
\cite{coh85} is concerned with the reducibility of some related
polynomials.  \cite{coh90b, coh90c} investigates their connection to
exceptional polynomials and coins the term ``sub-linearized''; see
also \cite{cohmat94}.  \citet*{couhav04} derive the number of
indecomposable subadditive polynomials and present an algorithm to
decompose subadditive polynomials.

 \citet[Theorem 3]{ore33b} describes exactly the right components of
 degree $p$ of an additive polynomial.  \cite{henmat99} relate such
 additive decompositions to subadditive polynomials, and in their
 Theorems~3.4 and 3.8 describe the collisions of \autoref{thm:nonadd}
 below.  The polynomials of \autoref{thm:constmulti} popped up in the
 course of trying to prove that these examples might be the only ones;
 see the proof of \autoref{thm:normal}.  In
 \autoref{sec:classification}, we show that together with the
 Frobenius collisions of \autoref{exa:frob}, these two examples and their
 shifts comprise all 2-collisions at degree $p^{2}$.

\begin{fact}
\label{thm:nonadd}
Let $r$ be a power of $p$, $u, s\in \ff^{\times}$, $\varepsilon \in \{0,1\}$, $m$ a positive divisor of $r-1$, $\ell = (r-1)/m$, and
\begin{equation}
 \label{eq:7}
\begin{split}
  f &= \SimpConst{u,s,\varepsilon,m} = x (x^{\ell(r+1)}-\varepsilon u
  s^{r}x^{\ell} + us^{r+1})^{m} \in P_{r^{2}}(\ff), \\
 T &= \{t \in \ff \colon t^{r+1} -\varepsilon ut + u = 0\}.
  \end{split}
\end{equation}
For each $t \in T$ and
\begin{equation}
\label{eq:80}
\begin{split}
g   & = x (x^{\ell}-us^{r}t^{-1})^{m}, \\
h   & = x (x^{\ell}-st)^{m},
\end{split}
\end{equation}
both in $P_{r}(\ff)$, we have $f = g \circ h$.
Moreover, $f$ has a $\# T$-collision.
\end{fact}

The polynomials $f$ in
\eqref{eq:7} are ``simply original'' in the sense that they
have a simple root at $0$.  This motivates the designation $S$.
\begin{proof}
For $t \in T$, we have
  \begin{align}
    g \circ h & = x (x^{\ell} -st)^{m}(x^{\ell}(x^{\ell}-st)^{r-1}-us^{r}t^{-1})^{m} \\
& = x (x^{\ell}(x^{\ell}-st)^{r}-(x^{\ell}-st)us^{r}t^{-1})^{m} \\
& = x (x^{\ell r + \ell} - s^{r} t^{r} x^{\ell} - u s^{r} t^{-1}x^{\ell} + u s^{r+1})^{m} \\
& = x(x^{\ell(r+1)} -s^{r}(t^{r}+ut^{-1})x^{\ell}+us^{r+1})^{m} \\
& = x (x^{\ell(r+1)}-\varepsilon u s^{r}x^{\ell} + us^{r+1})^{m} =f.
  \end{align}
This proves that $(g,h)$ is a decomposition of $f$.  While $f$ does
not depend on $t$, the $\#T$ different choices for $t$ yield $\# T$
pairwise different values for the coefficients of $x^{r-\ell}$ in $h$, namely
\begin{equation}
  h_{r-\ell} = -mst \neq 0. \tag*{\qedhere}
\end{equation}
\end{proof}

The polynomial $\SimpConst{u,s,\varepsilon,m}$ is $r$-additive for $m
= 1$ and $(r, m)$-subadditive for all $m$.  \cite{blu04a} shows
that for $\varepsilon = 1$ and $F \cap \Fr$ of size $Q$, the cardinality of $T$ is either $0$,
$1$, $2$, or $Q+1$.
This also holds for $\varepsilon = 0$.  In either case, $T$ is
independent of $m$ and $\ell$.  If $T$ is empty, then
$\SimpConst{u,s,\varepsilon,m}$ has no decomposition of the form
\eqref{eq:80}, but $r+1$ such decompositions exist over the
splitting field of the squarefree polynomial $y^{r+1} -\varepsilon uy
+ u \in F[y]$.

For a polynomial $f \in P_n(F)$ and an integer $i$, we
denote the coefficient of $x^{i}$ in $f$ by $f_i$, so that $f = x^n +
\sum_{1 \leq i < n} f_i x^i$ with $f_{i} \in F$.  The \emph{second degree} of $f$ is
\begin{equation}
\label{eq:secdeg}
\deg_{2} f = \deg(f-x^{n}).
\end{equation}
 If $p \mid n$ and $p \nmid \deg_{2} f$, then $\deg_{2} f = \deg (f') + 1$.

\begin{fact}[\cite*{gatgie10a}, Proposition 6.2]
\label{lem:unique1}
 Let $r$ be a power of $p$, and $u$, $s$, $\varepsilon$, $m$ and $u^{*}$, $s^{*}$, $\varepsilon^{*}$, $m^{*}$ satisfy the conditions
 of \autoref{thm:nonadd}.  For $f = \SimpConst{u, s, \varepsilon, m}$
 and $f^{*} =  \SimpConst{u^{*}, s^{*}, \varepsilon^{*}, m^{*}}$, the following hold.
\begin{ronumerate}
\item\label{item:ii} For $\varepsilon = 1$, we have $f=f^{*}$ if and only if
     $(u, s, \varepsilon, m) = (u^{*}, s^{*}, \varepsilon^{*}, m^{*})$.
\item\label{item:iv} For $\varepsilon = 0$, we have $f = f^{*}$ if and
  only if $(us^{r+1}, \varepsilon, m) = (u^{*}(s^{*})^{r+1},
  \varepsilon^{*}, m^{*})$.
\item\label{item:4} The stabilizer of $f$ under original shifting is
  $F$ if $m=1$, and $\{0\}$ otherwise.  For $F = \Fq$, the orbit of $f$ under
  original shifting has size $1$ if $m=1$, and size $q$ otherwise.
\item\label{item:5} The only polynomial of the form \eqref{eq:7} in the
  orbit of $f$ under original shifting is $f$ itself.
\end{ronumerate}
\end{fact}

\begin{proof}
The appearance of $O(x^{i})$ for some integer $i$ in an equation means
the existence of some polynomial of degree at most $i$ that makes the
equation valid.

Let $\ell = (r-1)/m$.  Then $\gcd(r, \ell) = \gcd(r, m) = 1$ and $\ell
m \equiv -1 \bmod p$.  We have
\begin{align}
f & = x ( x^{\ell(r+1)} - \varepsilon u s^{r} x^{\ell} +
  us^{r+1})^{m} \\
  & =  x ( x^{r^{2}-1} - m \varepsilon u s^{r} x^{r^{2}-\ell r - 1} +
  mus^{r+1}x^{r^{2}-\ell r - \ell - 1} + O(x^{r^{2}-2\ell r - 1 })) \\
  & = x^{r^{2}} - m \varepsilon u s^{r} x^{r^{2}-\ell r} + m u s^{r+1}
x^{r^{2}-\ell r - \ell} + O(x^{r^{2}-2\ell r}), \label{eq:10} \\
 f_{r^{2}-\ell r} & = -m\varepsilon us^{r},   \label{eq:4} \\
 f_{r^{2} - \ell r-\ell} & = mus^{r+1} \neq 0, \label{eq:9} \\
\deg_{2} f & = \begin{cases}
r^{2} - \ell r & \text{ if } \varepsilon =1, \\
r^{2} -\ell r - \ell & \text{ if } \varepsilon = 0.
\end{cases} \label{eq:55}
\end{align}
From the last equation, we find $\varepsilon = 1$ if $r \mid \deg_{2} f $, and
$\varepsilon = 0$ otherwise. For either value of $\varepsilon$,
$\deg_{2} f$ determines $\ell$ and $m = (r-1)/\ell$ uniquely.
Similarly, $\deg_{2} f^{*}$ determines $\varepsilon^{*}$, $\ell^{*}$,
and $m^{*}$ uniquely.  Therefore, if $\deg_{2} f = \deg_{2} f^{*}$,
then
\begin{equation}
  \label{eq:15}
  (\varepsilon, \ell, m) = (\varepsilon^{*}, \ell^{*}, m^{*}).
\end{equation}
Furthermore, $m$ and the coefficient $f_{r^{2}-\ell r - \ell}$
determine $u s^{r+1} = f_{r^2-\ell r -\ell} /m$ uniquely by
\eqref{eq:9}.  Similarly, $m^{*}$ and $f^{*}_{r^{2}-\ell^{*}r -\ell^{*}}$ determine
$u^{*}(s^{*})^{r+1}$ uniquely.  Thus, if $m = m^{*}$ and $f_{r^{2}-\ell r
  - \ell} = f^{*}_{r^{2}-\ell^{*} r - \ell^{*}}$, then
\begin{equation}
  \label{eq:13}
us^{r+1}  = u^{*}(s^{*})^{r+1}.
\end{equation}

\ref{item:ii} If $(u,
s, \varepsilon, m) = (u^{*},
s^{*}, \varepsilon^{*}, m^{*})$, then $f=f^{*}$.  On
the other hand, we have $f_{r^2 - \ell r} = -mus^{r} \neq 0$ in
\eqref{eq:4} and with \eqref{eq:9} this determines uniquely
\begin{equation}
  \label{eq:58}
  \begin{split}
s & =  - f_{r^{2}-\ell r - \ell} /   f_{r^{2}-\ell r}, \\
u & =  - f_{r^{2} -\ell r} / ms^{r} = \ell f_{r^{2} -\ell r} / s^{r}.
  \end{split}
\end{equation}
This implies the claim \ref{item:ii}.

\ref{item:iv} The condition $(us^{r+1}, \varepsilon, m) = (u^{*}(s^{*})^{r+1}, \varepsilon^{*}, m^{*})$ is sufficient for $f
= f^{*}$ by direct computation from \eqref{eq:7}.  It is also
necessary by \eqref{eq:15} and \eqref{eq:13}.

\ref{item:4} For $m = 1$, $f$ is $r$-additive as noted after the proof of
\autoref{thm:nonadd} and $f^{[w]} = f$
for all $w \in F$.
For $m > 1$ and $w \in F$, we find
\begin{align}
\begin{split} \label{eq:10shift}
f^{[w]} &= x^{r^{2}} - m \varepsilon u s^{r} x^{r^{2}-\ell r} + m u
s^{r+1} x^{r^{2}-\ell r - \ell} \\
& \quad + w u s^{r+1}x^{r^{2}-\ell r - \ell
  -1} + O(x^{r^{2}-\ell r - \ell -2}),
\end{split}\\
f^{[w]}_{r^{2}-\ell r}  & = f_{r^{2}-\ell r} = -m \varepsilon u s^{r}, \label{eq:4w} \\
  f^{[w]}_{r^{2} - \ell r-\ell} & = f_{r^{2} - \ell r
    -\ell} = mus^{r+1} \neq 0, \label{eq:9w} \\
  f^{[w]}_{r^{2} - \ell r-\ell - 1} & = wus^{r+1}.   \label{eq:5}
\end{align}
We have $f = f^{[0]}$ by definition and $f \neq f^{[w]}$ for $w \neq
0$ by \eqref{eq:5} and $us^{r+1} \neq 0$.

\ref{item:5} For $m = 1$, the claim follows from \ref{item:4}.  For $m
> 1$ and $w \in F$, assume $f_{0} = \SimpConst{u_{0}, s_{0},
  \varepsilon_{0}, m_0} = f^{[w]}$ for parameters $u_{0}, s_{0},
\varepsilon_{0}, m_{0}$ satisfying the conditions of
\autoref{thm:nonadd}.  Then $\deg_{2}f_{0} = \deg_{2} f^{[w]}$ by
assumption and
\begin{equation}
\label{eq:55w}
\deg_{2} f^{[w]} = \deg_{2} f = \begin{cases}
r^{2} - \ell r & \text{ if } \varepsilon =1, \\
r^{2} -\ell r - \ell & \text{ if } \varepsilon = 0,
\end{cases}
\end{equation}
from \eqref{eq:10shift} and \eqref{eq:55}.  Thus, we have
$\ell = \ell_0$ by \eqref{eq:15}.  The coefficient of $x^{r^{2} -
  \ell r - \ell -1}$ is $0$ in $f_{0}$ and $wus^{r+1}$ in $f^{[w]}$
by \eqref{eq:10} and \eqref{eq:5}, respectively.  With
$us^{r+1} \neq 0$, we have $w = 0$ and $f_{0} = f^{[0]} = f$.
\end{proof}

\autoref{algo:recover-simple-paras} identifies the examples of
\autoref{thm:nonadd} and their shifts.  The algorithm involves
divisions which we execute conditionally ``if defined''.  Namely, for
integers the quotient is returned, if it is an integer, and for field
elements, if the denominator is nonzero. Otherwise, ``failure'' is
returned.  Besides the field operations $+$, $-$, $\cdot$, we assume a routine for computing the number of roots in $F$ of a polynomial.
Furthermore, we denote by $\MM(n)$ a number of field operations which
is sufficient to compute the product of
two polynomials of degree at most $n$.

\begin{algorithm2f}
\caption{Identify simply original polynomials}
\label{algo:recover-simple-paras}
\KwIn{a polynomial $f = \sum_i f_{i}x^{i} \in P_{r^{2}}(F)$ with all $f_{i} \in \ff$ and $r$ a power of $\chara \ff$}
\KwOut{integer $k$, parameters $u, s, \varepsilon, m$ as in \autoref{thm:nonadd}, and $w\in F$ such that $f = \SimpConst{u, s, \varepsilon, m}^{[w]}$ has a $k$-collision with $k=\# T$ as in \eqref{eq:7}, if such values exist, and ``failure'' otherwise}

\lIf{$\deg_{2} f = -\infty$}{\Return{``failure''}} \label{step:1}
\eIf{$r \mid \deg_{2}f$}{
  $\varepsilon \gets 1$\; \label{step:3}
  $\ell \gets (r^{2}-\deg_{2}f)/r$ and $m \gets  (r-1)/\ell$ if
  defined\; \label{step:2}
  $s \gets -f_{r^{2}-\ell r-\ell}/f_{r^{2}-\ell r}$ if defined\; \label{step:5}
}{
  $\varepsilon \gets 0$\; \label{step:8}
  $\ell \gets (r^{2}-\deg_{2}f)/(r+1)$ and $m \gets  (r-1)/\ell$ if
  defined\; \label{step:4}
  $s \gets 1$\; \label{step:10}
}
$u \gets -\ell f_{r^{2}-\ell r-\ell} / s^{r+1}$ if defined \; \label{step:11}
\lIf{$us= 0$}{\Return{``failure''}} \label{step:20}
$w \gets mf_{r^{2}-\ell r - \ell - 1}/ f_{r^{2}-\ell r - \ell}$ if defined\; \label{step:16}
\If{$f = \SimpConst{u, s, \varepsilon, m}^{[w]}$ \label{step:18}}{
$k \gets \# \{y \in F \colon y^{r+1} - \varepsilon u y + u = 0\}$\label{step:19}\;
  \KwRet{$k, u, s, \varepsilon, m, w$}\;
}
\KwRet{``failure''}
\end{algorithm2f}
\begin{theorem}
\label{thm:algo1}
   \autoref{algo:recover-simple-paras} works correctly as
   specified. If $F=\Fq$, it  takes $O(\MM(n) \log (nq))$
   field operations on input a polynomial of degree $n=r^2$.
\end{theorem}
\begin{proof}
  For the first claim,
  we show that for $\alt{u}, \alt{s}, \alt{\varepsilon}, \alt{m}$ as in \autoref{thm:nonadd} and $\alt{w}\in F$ the algorithm does not fail on input $f = \SimpConst{\alt{u}, \alt{s}, \alt{\varepsilon}, \alt{m}}^{[\alt{w}]}$.

  We have $\deg_{2}f > 0$ by \eqref{eq:55w}.  Thus, \autoref{step:1}
  does not return ``failure''.  By the same equation, we have $r \mid
  \deg_{2} f$ if and only if $\varepsilon_{0} = 1$.  Therefore,
  $\varepsilon = \alt{\varepsilon}$ after step~\ref{step:3} or
  \ref{step:8}, respectively, and since \eqref{eq:55w} determines
  $\alt{\ell}=(r-1)/\alt{m}$ uniquely, we find $\ell = \alt{\ell}$ and $m =
  (r-1)/\alt{\ell} = \alt{m}$ after step \ref{step:2} or \ref{step:4},
  respectively.  If $\varepsilon = 1$, then step \ref{step:5} computes
  $s = \alt{s}$ from \eqref{eq:58}, \eqref{eq:4w}, and \eqref{eq:9w}.
  Furthermore, step \ref{step:11} computes $u = \alt{u}$ from
  \eqref{eq:9} and \eqref{eq:9w}.  If $\varepsilon  = 0$, then
\begin{equation}
\label{eq:18}
\SimpConst{\alt{u}, \alt{s}, 0, m}^{[\alt{w}]} = (x(x^{\ell(r+1)} + u_0 s_0 ^{r+1})^{m})^{[\alt{w}]} = \SimpConst{u_0 s_0 ^{r+1},  1, 0, m}^{[\alt{w}]}.
\end{equation}
Therefore, we can choose $s = 1$ in \autoref{step:10} and set $u = - \ell f_{r^2 - \ell r
  -\ell} = u_0 s_0^{r+1}$ by \eqref{eq:9} and \eqref{eq:9w} in \autoref{step:11}. For
either value of $\varepsilon$, we
have $u s \neq 0$ from $u_{0}s_{0} \neq 0$ and \autoref{step:20} does not return ``failure''.

For $m = 1$, we have
\begin{equation}
  \label{eq:11}
  \SimpConst{u, s, \varepsilon, 1}^{[w_{0}]} =
  \SimpConst{u, s, \varepsilon, 1}^{[0]}
\end{equation}
by \autoref{lem:unique1}~\ref{item:4} and $w = f_{0} / f_{1} = 0$ in \autoref{step:16}
is a valid choice.  For $m>1$, we find $w_{0}$ from \eqref{eq:9w} and \eqref{eq:5} as
\begin{equation}
  \label{eq:12}
  w = m f_{r^{2}-\ell r - \ell - 1}/ f_{r^{2}-\ell r - \ell} = w_0.
\end{equation}
A polynomial $f$ of
the assumed form passes the final test in step \ref{step:18}, while an
$f$ not of this form will fail here at the latest.  The size $k$ of the
set $T = \{t \in \ff \colon t^{r+1} -\varepsilon ut + u = 0\}$ is
computed in step \ref{step:19} and $f$ is a $k$-collision according to \autoref{thm:nonadd}.

In the following cost estimate for $\ff = \Fq$, we ignore the (cheap) operations on integers.
The calculation of the right-hand side in step~\ref{step:18} takes
$O(\MM(n) \log n)$ field operations, and the test another $n$
operations.
In step~\ref{step:19}, we compute $k$ as $\deg_{y} (\gcd ( y^{q} -
y, y^{r+1} - \varepsilon u y + u))$ with $O(\MM(r)(\log q + \log r))$ field operations.
The cost of all other steps is dominated by these bounds.
\end{proof}

Let $C_{n,k}^{(S)}(\ff)$ denote the set of polynomials in $P_{n}(\ff)$
that are shifts of some $\SimpConst{u, s, \varepsilon, m}$ with $T$ as in \eqref{eq:7} of
cardinality $k$.
Over a finite field, $\# C_{r^{2},k}^{(S)}(\Fq)$ can be computed
exactly, as in \citet*[Corollary 6.3]{gatgie10a}.

\begin{proposition}
\label{fac:simply_count}
Let $r$ be a power of $p$, $q$ a power of $r$, and $\divfct$ the number of positive divisors of $r-1$.  For $k \geq 2$, we have
\begin{equation}
\# C_{r^{2},k}^{(S)}(\Fq) = \begin{cases*}
\dfrac{(\divfct q - q +1)(q-1)^{2}(r-2)}{2(r-1)} & if $k = 2$, \\
\dfrac{(\divfct q - q +1)(q-1)(q-r)}{r(r^{2}-1)} & if $k = r+1$, \\
0 & otherwise.
\end{cases*}
\end{equation}
\end{proposition}

\begin{proof}
We count the polynomials in $C_{r^{2},k}^{(S)}(\Fq)$ by counting the
admissible parameters $u$, $s$, $\varepsilon$, $m$, $w$ modulo the ambiguities described in \autoref{lem:unique1}.

For $\varepsilon = 1$, we count the possible $u \in
\Fq^{\times}$ such that $y^{r+1} - uy + u
\in \Fq [y]$ has exactly $k$ roots in $\Fq$.  Let $a,b \in
  \Fq^{\times}$ and $u = a^{r+1}b^{-r}$.  The invertible transformation $x
  \mapsto y= -ab^{-1}x$ gives a
  bijection
\begin{equation}
\label{eq:16}
 \{ x \in
    \Fq^{\times} \colon x^{r+1} + a x + b = 0\}  \leftrightarrow  \{ y \in \Fq^{\times} \colon  y^{r+1} - uy + u = 0\}.
\end{equation}
Theorem~5.1 and Proposition~5.4 of \cite*{gatgie10a} determine the
number $c_{q,r,k}^{(2)}$ of pairs $(a,b) \in (\Fq^{\times})^{2}$ such
that $x^{r+1} + a x + b$ has exactly $k$ roots, as described below.
Every value of $u$ corresponds to exactly $q-1$ pairs $(a,b)$, namely
an arbitrary $a \in \Fq^{\times}$ and $b$ uniquely determined by
$b^{r} = u^{-1}a^{r+1}$.  Hence, there are exactly
$c_{q,r,k}^{(2)}/(q-1)$ values for $u$ where $\# T = k$.  For $m=1$,
the orbit under original shifting has size $1$ by
\autoref{lem:unique1}~\ref{item:4} and taking into account the $q-1$
possible choices for $s$ we find that there are $c_{q,r,k}^{(2)}$
polynomials of the form $\SimpConst{u, s, 1, 1}^{[w]}$.  For $m>1$,
the orbit under original shifting contains exactly one polynomial of
the form \eqref{eq:7} by \autoref{lem:unique1}~\ref{item:5} and has
size $q$ by \ref{item:4}.  Taking into account the $q-1$ choices for
$s$ and the $\tau - 1$ possible values for $m$, we find that there are
$c_{q,r,k}^{(2)} \cdot (\divfct-1) \cdot q$ polynomials of the form
$\SimpConst{u, s, 1, m}^{[w]}$.

For $\varepsilon = 0$, we have $\SimpConst{u,s,0,m}^{[w]} =
\SimpConst{us^{r+1},1,0,m}^{[w]}$ as in \eqref{eq:18} and $ T = \{t \in \Fq \colon t^{r+1} + us^{r+1} =
0\}$ as in \eqref{eq:7}.  This set
has exactly $\gamma = \gcd(r+1,q-1)$ elements, if $-u$ is an $(r+1)$st
power, and is empty otherwise.  Then $\# T = k \geq 2$ if and only if
$k = \gamma$ and $-u$ is an $(r+1)$st power.  There are exactly
$(q-1)/\gamma$ distinct $(r+1)$st powers in $\Fq^{\times}$ and
therefore exactly $(q-1)/\gamma$ distinct values for $us^{r+1}$ such
that $\# T = \gamma$.  With $\delta$ being Kronecker's delta function,
we find, as above, that there are $\delta_{\gamma,k} \cdot (q-1) / \gamma$ polynomials of the
form $\SimpConst{u,s,0,1}^{[w]}$ in $C_{r^{2},k}^{(S)}(\Fq)$ and $
\delta_{\gamma,k} \cdot (\tau-1)q(q-1) / \gamma$ of the form
$\SimpConst{u,s,0,m}^{[w]}$ with $m>1$.

This yields
\begin{equation}
\# C_{r^{2},k}^{(S)}(\Fq)  =  \left(\divfct q  -q  + 1\right) \cdot \left( c_{q,r,k}^{(2)} + \delta_{\gamma,k} \frac{
        q-1}{\gamma}  \right).
\end{equation}
The work cited above provides the following explicit expressions for
$k \geq 2$, with $q = r^{d}$:
\begin{align}
  c_{q,r,2}^{(2]} & =
  \begin{cases}
    \dfrac{(q-1)(qr-2q-2r+3)}{2(r-1)} & \text{if $q$ and $d$ are odd}, \\
\dfrac{(q-1)^{2}(r-2)}{2(r-1)} & \text{otherwise},
  \end{cases} \\
  c_{q,r,r+1}^{(2)} & =
  \begin{cases}
    \dfrac{(q-1)(q-r^{2})}{r(r^{2}-1)} & \text{if $d$ is even}, \\
\dfrac{(q-1)(q-r)}{r(r^{2}-1)} & \text{if $d$ is odd},
  \end{cases}
\end{align}
and $c_{q,r,k}^{(2)} = 0$ for $k \notin \{2, r+1\}$.
Furthermore, we have from Lemma 3.29 in \citet[Preprint]{gat09b}
  \begin{equation}
    \gamma = \gcd(r+1, r^{d}-1) = \begin{cases}
1 & \text{if $d$ is odd and $r$ is even}, \\
2 & \text{if $d$ is odd and $r$ is odd}, \\
r+1 & \text{if $d$ is even}.
\end{cases}
\end{equation}
The claimed formulas follow from
\begin{equation}
  c_{q,r,k}^{(2)} + \delta_{\gamma,k} \frac{q-1}{\gamma} = \left\{ \begin{alignedat}{3}
& \frac{(q-1)^{2}(r-2)}{2(r-1)} && \text{ if $k=2$}, && \\
& \frac{(q-1)(q-r)}{r(r^{2}-1)} && \text{ if $k=r+1$},&& \\
& 0 && \text{ otherwise}. && \qedhere
\end{alignedat}
\right.
\end{equation}
\end{proof}

For a prime $p\geq 7$, we have $\tau \geq 4$.
Large values of $\tau$ occur when $m \approx \exp(k \log k)$ is
the product of the first $k$ primes and $p \leq m^{C}$ the smallest
prime congruent $1 \bmod m$  for Linnik's
constant $C$.  Then $k \approx \log m/\loglog m
\gtrsim C^{-1} \log p/\loglog p$  and
$ \tau \geq 2^k \gtrsim 2^{ C^{-1} \log p/\loglog p }$.  By
\cite{hea92} and \cite{xyl11} we can take $C$ just under $5$.  Except for the
constant factor, $\tau$ is asymptotically not more than this
value \cite[Theorem~317]{harwri85}.  \cite{lucshp08} give
general results on the possible values of $\tau$.  It follows
that $\# C_{r^{2},2}^{(S)}(\Fq) \approx \tau q^{3}/2$ is in
$q^{3} O(p^{1/\loglog p})$.

The odd/even distinctions for $q$, $r$, and $d$ cancel out in the formula
of \autoref{fac:simply_count}.  This might indicate that those
distinctions are alien to the problem.

The second and new construction of collisions goes as follows.
\begin{theorem}
\label{thm:constmulti}
  Let $r$ be a power of $p$, $b \in \ff^\times$, $a \in \ff \setminus \{0, b^{r}\}$, $a^* = b^r - a$, $m$ an integer with $1 < m <r-1$ and $p \nmid m$,  $m^* = r - m$, and
  \begin{equation}
  \label{eq:3normal}
  \begin{split}
   f = \MultConst{a, b, m} &= x^{m m^*} (x - b)^{m m^*} \left(x^m + a^* b^{-r} ((x-b)^{m} - x^m)\right)^m \\
	&\quad\quad\quad \cdot
        \left(x^{m^*} + a b^{-r} ((x-b)^{m^*} - x^{m^*})
        \right)^{m^*},\\
g &= x^m (x - a)^{m^*}, \\
h &= x^{r} + a^{*}b^{-r}(x^{m^*}(x-b)^m - x^{r}),  \\
g^* &= x^{m^*} (x - a^*)^m, \\
h^* &= x^{r} + ab^{-r}(x^m (x-b)^{m^*} - x^{r}).
\end{split}
\end{equation}
Then $f = g \circ h = g^{*} \circ h^{*} \in P_{r^{2}}(\ff)$ has a 2-collision.
\end{theorem}

The polynomials $f$ in
\eqref{eq:3normal} are ``multiply original'' in the sense that they
have a multiple root at $0$.  This motivates the designation $M$.
The notation is set up so that ${}^{*}$ acts as an involution on our
data, leaving $b$, $f$, $r$, and $x$ invariant.

Mike \cite{zie11} points out that the rational functions of case (4) in Proposition 5.6 of \cite{avazan03} can be transformed into \eqref{eq:3normal}.  Zieve also mentions that this example already occurs in unpublished work of his, joint with Bob Beals.

\begin{proof}
Let
\begin{equation}
  \label{eq:48}
\begin{split}
  \qepol & = h / x^{m^*} = x^m + a^* b^{-r}((x-b)^m - x^m), \\
\qepol^{*} & = h^{*}/x^m = x^{m^*} + a b^{-r}((x-b)^{m^*} - x^{m^*}).
\end{split}
\end{equation}
Then $h-a = (x-b)^m \qepol^{*}$ and $h^{*} - a^{*} = (x-b)^{m^*} \qepol$.  It follows that
\begin{equation}
  \label{eq:2}
  g \circ h = g^{*} \circ h^{*} = x^{m m^*} (x-b)^{m m^*} \qepol^m (\qepol^{*})^{m^*} = f.
\end{equation}
If $g = g^*$, then the coefficients of $x^{r-1}$ in $g$ and $g^{*}$
yield $mb^{r} = 0$, hence $p \mid m$, a contradiction. Thus $f$ is a 2-collision.
\end{proof}

For $r \leq 4$, there is no value of $m$ satisfying the assumptions.
The construction works for arbitrary  $a \in \ff$ and $1 \leq m \leq r-1$.
But when $a \in \{0,b^r\}$, we get a Frobenius collision; see \autoref{exa:frob}.
When $p \mid m$, we write $m = p^e m_0$ with $p \nmid m_0$ and have $f = x^{p^e} \circ \MultConst{a, b^{p^e} \!, m_0} \circ x^{p^e}$ with $r/{p^e}$ instead of $r$ in \eqref{eq:3normal}.
When $m$ is $1$ or $r-1$, an original shift of \eqref{eq:3normal}
yields a polynomial of the form $\SimpConst{u, s, \varepsilon, m}$.  Indeed, for $m = 1$,
let $w = a^*b^{-r+1}$, $c =  (ab^{1-r})^r - a^*$, and
 \begin{equation}
   (u, s, \varepsilon, m, t, t^*) = \begin{cases*}
(-aa^*b^{1-r}, 1, 0, r-1, -a^*b^{1-r}, ab^{1-r}) & if $c = 0$, \\
(c/ s^{r}, -aa^*b^{1-r} / c, 1, r-1, -c/a, c/a^*) & otherwise.
  \end{cases*}
  \end{equation}
Then $M(a,b,1)^{[w]} = S(u,s,\varepsilon,m)$, $(g,h)^{[w]}$ is of
the form \eqref{eq:80}, and so is $(g^*, h^*)^{[w]}$ with $t$ replaced by
$t^*$. Furthermore, for $m = r-1$, we have $M(a,b,r-1)=M(a^*,b,1)$ and the claimed parameters can be found as described by interchanging $a$ and $a^*$.

Next, we describe the (non)uniqueness of this construction. We take all polynomial gcds to be monic, except that $\gcd(0,0) = 0$.

\begin{proposition}
\label{pro:multi_uniqueness}
 Let $r$ be a power of $p$, $b \in \ff^\times$, $a \in \ff \setminus
 \{0, b^{r}\}$, $m$ an integer with $1 < m <r-1$ and $p \nmid m$, and $f = \MultConst{a,b,m}$ as in \eqref{eq:3normal}.
 Then the following hold.
 \begin{ronumerate}
 \item\label{item:-1} In the notation of \autoref{thm:constmulti} and with $H$ and $H^{*}$ as in \eqref{eq:48}, we
   have $\gcd(m, m^*)=1$ and
   the four polynomials $x$, $x-b$,
   $\qepol$, and $\qepol^*$ are squarefree and pairwise coprime.
 \item\label{item:1} The stabilizer of $f$ under original shifting is
   $\{ 0\}$. For $\ff = \Fq$, the orbit of $f$ under original shifting has size $q$.
 \item\label{item:0} For $\alt{a},\alt{b},\alt{m}$ satisfying  the conditions of \autoref{thm:constmulti}, we have  $\MultConst{a,b,m} =\MultConst{\alt{a},\alt{b},\alt{m}}$ if and only if $(\alt{a},\alt{b},\alt{m}) \in \{ (a,b,m), (a^{*},b, m^{*})\}$. If we impose the additional condition $m < r/2$, then $(a,b,m)$ is uniquely determined by $\MultConst{a,b,m}$.
 \item\label{item:2} There are exactly two polynomials of the form
   \eqref{eq:3normal} in the orbit of $f$ under original shifting, namely $f$ and $f^{[b]}=\MultConst{-a^*, -b, m}$.
 \end{ronumerate}
\end{proposition}

\begin{proof}
\ref{item:-1} If $d>1$ was a common divisor of $m$ and $m^*$, then $d \mid m + m^* = r$ and thus $d$ would be a power of $p$---in particular $p \mid m$, a contradiction. Thus $\gcd(m, m^*) = 1$.
From $m \qepol - x \dif{\qepol}  = m^{*} a^* b^{1-r} (x-b)^{m-1}$ and
$\qepol (0) \cdot \qepol ( b) \neq 0$, we find that $\qepol$ is
squarefree and coprime to $x(x-b)$, and similarly for $\qepol^{*}$.
Since $\qepol \mid h$, $\qepol^* \mid (h-a)$, and $\gcd(h, h-a) =1$, we have $\gcd(\qepol, \qepol^*) = 1$.

\ref{item:1} For the coefficient of $x^{r^2 - r- 2}$ in the composition $f = g \circ h$, we find
\begin{equation}
  \label{eq:6}
  f_{r^2-r-2} = g_{r-1} ( h_{r-1}^{2} - h_{r-2}),
\end{equation}
since $r>2$.  For the shifted composition $f^{[w]} = g^{[h(w)]} \circ h^{[w]}$, we have the coefficients
\begin{align}
  g_{r-1}^{[h(w)]} &= g_{r-1}= -m^* a\neq 0, \\
   h_{r-1}^{[w]} &= h_{r-1}= -m a^{*}(-b)^{1-r} \neq 0, \\
   h_{r-2}^{[w]} &= h_{r-2} - wh_{r-1}, \\
   f_{r^{2}-r-2}^{[w]} & = g_{r-1} ( h_{r-1}^{2} - h_{r-2} + wh_{r-1}).
\end{align}
Thus, $f_{r^2-r-2} = f_{r^2-r-2}^{[w]}$ if and only if $w=0$.

\ref{item:0}
Sufficiency is a direct computation.  Conversely, assume that $f =
\MultConst{a,b,m} =\MultConst{\alt{a},\alt{b},\alt{m}} = \alt{f}$.
From \ref{item:-1} and the multiplicity $mm^{*}$ of $0$ and $b$ in
$f$, we find $m m^*  = \alt{m} \alt{m^*}$ and $\alt{b} = b$; see
\eqref{eq:2}.  If necessary, we replace $(a,b,m)$ by $(a^{*}, b, m^*)$, and obtain $\alt{m} = m$.  Dividing $f$ and $\alt{f}$ by $x^{m m^*} (x-b)^{m m^*}$ yields $\qepol^m (\qepol^*)^{m^*} = \alt{\qepol^m} (\alt{\qepol^*})^{m^*} $ by \eqref{eq:2}. Hence by \ref{item:-1}, we find $\alt{\qepol} = \qepol$ and thus $\alt{a} = a$.

\ref{item:2}
We find $f^{[b]} = \MultConst{-a^*, -b, m}$ by a direct computation.
Conversely, we take $a_0$, $b_0$, $m_0$ as in \autoref{thm:constmulti} and assume that
$f^{[w]} = \MultConst{\alt{a},\alt{b},\alt{m}} = \alt{f}$.
By \ref{item:0}, we may assume that $m, \alt{m} < r/2$. We have
\begin{equation}
\label{eq:60}
\begin{split}
\dif{g} &= m^* a x^{m-1} (x-a)^{m^*-1}, \\
\dif{h} &= m a^* b^{1-r} x^{m^*-1} (x-b)^{m-1} , \\
\dif{f} &= (\dif{g} \circ h) \cdot \dif{h}\\
 &= mm^{*} a a^* b^{1-r} (x(x-b))^{m m^* -1} \qepol^{m-1} (\qepol^*)^{m^* -1}.
 \end{split}
 \end{equation}

Now \ref{item:-1} and $p \nmid m m^*$ show that $f'$ has roots of multiplicity $m m^* - 1$ exactly  at $0$ and $b$ and otherwise only roots of multiplicity at most $m^* -1 < mm^* -1$. Furthermore, $(f^{[w]})' =  \dif{f}(x+w)$ has roots of multiplicity $m m^* -1$ exactly at $-w$ and $b-w$.
Similarly, $\alt{f}$ has roots of multiplicity $\alt{m} \alt{m^*} -1$ at $0$ and $\alt{b}$, and all other roots have smaller multiplicity. It follows that $mm^* = \alt{m} \alt{m^*}$ and $m = \alt{m}$. Furthermore, one of $-w$ and $b -w$ equals $0$, so that $w \in \{0, b\}$. Hence $(a_0,b_0, m_0, w) \in \{ (a,b,m,0), (a^*,b,m^*,0), (-a^*, -b, m, b), (-a , -b, m^*, b)\}$. \qedhere
\end{proof}

We now provide the exact number of these collisions over $\Fq$, matching \autoref{fac:simply_count}.
When $r \leq 4$, there are no polynomials of the form \eqref{eq:3normal}.
\begin{corollary}
\label{cor:multi_count}
  For $r \geq 3$ and $F = \Fq$, the number of polynomials that are of the form \eqref{eq:3normal} or shifts thereof is
  \begin{equation}
    \label{eq:30}
   \frac{q(q-1)(q-2)(r-\frac{r}{p}-2)}{4}.
  \end{equation}
\end{corollary}

\begin{proof}
  There are $q-1$, $q-2$, and $r-r/p-2$ choices for the parameters  $b$, $a$, and $m$, respectively.  By \autoref{pro:multi_uniqueness}~\ref{item:0}, exactly two distinct triples of parameters generate the same polynomial \eqref{eq:3normal}. By \ref{item:1}, the shift orbits are of size $q$ and by \ref{item:2}, they contain two such polynomials each.
\end{proof}

Over a field $\ff$ of characteristic $p >0$, \autoref{algo:recover-multi-paras-infinite-fields} finds the parameters for polynomials that are  original shifts of \eqref{eq:3normal}, just as \autoref{algo:recover-simple-paras} does for original shifts of \eqref{eq:7}.  It involves conditional divisions and routines for extracting $p$th and square roots.  Given a field element, the latter produce a root, if one exists, and ``failure'' otherwise.  If $\ff$ is finite, then every element has a $p$th root.
The algorithm for a square root yields a subroutine to determine the set of roots of a quadratic polynomial.

\begin{algorithm2f}
\caption{Identify multiply original polynomials}
\label{algo:recover-multi-paras-infinite-fields}
\SetKw{Or}{or}
\KwIn{a polynomial $f \in P_{r^{2}}(\ff)$ with $r$ a power of $p = \chara \ff$}
\KwOut{parameters $a, b, m$, as in \autoref{thm:constmulti}, and $w \in \ff$ such that $f = \MultConst{a, b, m}^{[w]}$, if such values exist, and ``failure'' otherwise}

$\fnull \gets f' / \operatorname{lc}(f')$ if defined \; \label{step:i-1}
\lIf{$p=2$}{$\fnull \gets \fnull^{1/2}$ if defined} \label{step:i00}
$\feins \gets \fnull/\gcd(\fnull, \fnull')$ if defined \; \label{step:i01}
\lIf{$\deg \feins < 4$ \Or $\deg \feins > r + 2$}{\KwRet{``failure''}} \label{step:i100}
determine the maximal $k$ such that $\feins^{k} \mid \fnull$ via the generalized Taylor expansion of $\fnull$ in base $\feins$ \;  \label{step:i02}
\lIf{$p=2$}{$k \gets 2k$} \label{step:i06}
$m \gets \min\{k+1, r-k-1\}$\; \label{step:i03}
\lIf{$m < 2$}{\KwRet{``failure''}} \label{step:i101}
\eIf{$p = 2$ or $p \nmid m^{2}+1$}{ \label{step:i07}
  $\fzwei \gets \gcd( \feins^{r-m}, \fnull) / \gcd (\feins^{r-m-1}, \fnull)$\; \label{step:i05}
}{
  $\fdrei \gets \fnull/ \gcd (\feins^{r-m-1}, \fnull)$ if defined\; \label{step:i12}
  determine the maximal $\ell$ such that $p^{\ell}$ divides every exponent of $x$ with nonzero coefficient in $\fdrei$ \; \label{step:i17}
  $\fdrei \gets \fdrei^{1/p^{\ell}}$ if defined \; \label{step:i13}
  $\fzwei \gets \fdrei/\gcd(\fdrei,\fdrei')$ \; \label{step:i14}
}

\lIf{$\deg \fzwei \neq 2$}{\KwRet{``failure''}} \label{step:i15}
compute the set $X$ of roots of $\fzwei$ in $\ff$\; \label{step:i21}
\lIf{$\# X < 2$}{\KwRet{``failure''}} \label{step:i24}
write $X$ as $\{x_1, x_2\}$ and set $b \gets x_{2} - x_{1}$ and $w \gets -x_{1}$\; \label{step:i22}
compute the set $A$ of roots of $ y^2 - b^r y - m^{-2} b^{r-1}
\operatorname{lc}(f') \in \ff[y]$ in $\ff$\; \label{step:i23}
\For{$a \in A$}{
  \If{$f = \MultConst{a,b,m}^{[w]}$}{ \label{step:i10}
    \KwRet{$a ,b ,m, w$}
  }
}
\KwRet{``failure''}
\end{algorithm2f}
\begin{theorem}
\label{thm:algo2infcorrect}
\autoref{algo:recover-multi-paras-infinite-fields} works correctly as specified.  If $F=\Fq$, it takes $O(\MM(n)\log n + n \log q)$ field operations on input a polynomial of degree $n = r^2$.
\end{theorem}
\begin{proof}
For the correctness, it is sufficient---due to the check in \autoref{step:i10}---to show that for $\alt{a}, \alt{b}, \alt{m}$ as in \autoref{thm:constmulti} and $\alt{w} \in \ff$, the algorithm does not return ``failure''  on input $f = \MultConst{\alt{a}, \alt{b}, \alt{m}}^{[\alt{w}]}$.
As remarked after \autoref{thm:constmulti}, we have $r \geq 5$ and by \autoref{pro:multi_uniqueness}~\ref{item:0}, we may assume $\alt{m}<r/2$.
Furthermore, \eqref{eq:60} determines $\operatorname{lc}(f') \neq 0$ explicitly and \autoref{step:i-1} is defined. The square root in step~\ref{step:i00} is defined, since for $p=2$, $m_0$ and $r - m_0$ are odd and all exponents in the monic version of \eqref{eq:60} are even.

By \eqref{eq:60} and \autoref{pro:multi_uniqueness}~\ref{item:-1}, we have after steps~\ref{step:i-1} and \ref{step:i00}
\begin{equation}
  \label{eq:63}
  \fnull = \begin{cases*}
\varphi^{\alt{m}(r-\alt{m})-1} H_{0}^{\alt{m}-1} {H_{0}^{*}}^{r-\alt{m}-1} & if $p>2$, \\
\varphi^{(\alt{m}(r-\alt{m})-1)/2} H_{0}^{(\alt{m}-1)/2} {H_{0}^{*}}^{(r-\alt{m}-1)/2} & if $p=2$,
\end{cases*}
\end{equation}
with $\varphi = (x+\alt{w})(x-b_{0}+\alt{w})$, $H_{0} = H \circ (x+\alt{w})$, $H_{0}^{*} = H^{*} \circ (x+\alt{w})$, and $H$ and $H^{*}$ as in \eqref{eq:48} with $\alt{a}, \alt{a^*}, \alt{b}, \alt{m}, \alt{m^*}$ instead of $a, a^*, b, m, m^*$, respectively. By \autoref{pro:multi_uniqueness}~\ref{item:-1}, these three polynomials are squarefree and pairwise coprime.   Let $\delta$, $\varepsilon$, $\varepsilon^{*}$ be 0 if $p$ divides the exponent of $\varphi$, $H_{0}$, $H_{0}^{*}$, respectively, in \eqref{eq:63}, and be 1 otherwise.  Then
\begin{equation}
  \gcd (\fnull, \fnull') = \begin{cases*}
\varphi^{\alt{m}(r-\alt{m})-1-\delta} H_{0}^{\alt{m}-1-\varepsilon} {H_{0}^{*}}^{r-\alt{m}-1-\varepsilon^{*}} & if $p>2$, \\
\varphi^{(\alt{m}(r-\alt{m})-1)/2-\delta} H_{0}^{(\alt{m}-1)/2-\varepsilon} {H_{0}^{*}}^{(r-\alt{m}-1)/2-\varepsilon^{*}} & if $p=2$.
\end{cases*}
\end{equation}
This gcd is nonzero, and \autoref{step:i01} computes
\begin{equation}
\label{eq:8}
 \feins = \fnull / \gcd(\fnull, \fnull') = \varphi^{\delta} H_{0}^{\varepsilon} {H_{0}^{*}}^{\varepsilon^{*}}.
\end{equation}
We have
\begin{equation}
  \label{eq:62}
  \delta = \begin{cases*}
1 & if $p=2$ or $p \nmid \alt{m}^{2}+1$, \\
0 & otherwise.
\end{cases*}
\end{equation}
For odd $p$, this follows from $\alt{m}(r-\alt{m})-1 \equiv
-\alt{m}^{2}-1 \bmod p$, and for $p=2$ from $4 \nmid \alt{m^{2}}+1$.
The sum of the exponents of $H_{0}$ and $H_{0}^{*}$ in \eqref{eq:63} is $r-2$ for odd $p$ and $r/2-1$ for $p=2$.  In either case, it is coprime to $p$ and at least one of $\varepsilon$ and $\varepsilon^*$ equals $1$.  If $p>2$ and $\varepsilon=0$, then $\alt{m} \equiv 1 \bmod p $, and thus $\alt{m}^{2} \equiv 1 \bmod p$.  Hence $p \nmid \alt{m}^{2}+1$ and $\delta=1$.  Similarly, $\varepsilon^{*} =0$ implies $\delta=1$, and we find that at least two of $\delta$, $\varepsilon$, and $\varepsilon^{*}$ take the value 1. This also holds for $p=2$.

Since $\deg \varphi = 2$, $\deg H_{0}, \deg H_{0}^{*} \geq 2$, and $\deg H_0 + \deg H_0^* = r$, this implies $4 \leq \deg \feins \leq r+2$ and \autoref{step:i100} does not return ``failure''. The exponents in \eqref{eq:63} satisfy $m_0 -1 < r - m_0 - 1 < m_0(r - m_0) - 1$. If $p > 2$, then $k$ as determined in \autoref{step:i02} equals $\alt{m}-1$ if $\varepsilon=1$, and $r-\alt{m}-1$ otherwise.  In characteristic $2$, \autoref{step:i06} modifies $k \in \{(\alt{m}-1)/2, (r-\alt{m}-1)/2\}$, so that in any characteristic, \autoref{step:i03} recovers $m = \alt{m} \geq 2$ and \autoref{step:i101} does not return ``failure''.

The condition in \autoref{step:i07} reflects the case distinction in \eqref{eq:62}.
\begin{itemize}
\item If the condition holds, we have $\delta = 1$ and
  \begin{align}
    \gcd (\feins^{r-m}, \fnull) & = \varphi^{r-m} H_{0}^{\varepsilon(m-1)} {H_{0}^{*}}^{\varepsilon^{*}(r-m-1)}, \\
    \gcd (\feins^{r-m-1}, \fnull) & = \varphi^{r-m-1} H_{0}^{\varepsilon(m-1)} {H_{0}^{*}}^{\varepsilon^{*}(r-m-1)},
  \end{align}
and therefore $\fzwei = \varphi$ in \autoref{step:i05}.
\item Otherwise, we have $\delta = 0$, $p>2$, $\varepsilon = \varepsilon^{*} = 1$,
  \begin{align}
\fnull & = \varphi^{m(r-m)-1} H_{0}^{m-1} {H_{0}^{*}}^{r-m-1}, \\
\gcd (\feins^{r-m-1}, \fnull) & = H_{0}^{m-1} {H_{0}^{*}}^{r-m-1},
  \end{align}
and $\fdrei = \varphi^{m(r-m)-1}$ in \autoref{step:i12}.  After \autoref{step:i13}, we have $\fdrei = \varphi^{e}$ for some $e$ with $p \nmid e$ and $\fzwei=\varphi^{e} / \varphi^{e-1} = \varphi$ in \autoref{step:i14}.
\end{itemize}

In any case, we have $\fzwei = (x+\alt{w})(x-\alt{b}+\alt{w})$ with
distinct roots $-\alt{w}$ and $\alt{b}-\alt{w}$ in $\ff$, and
steps~\ref{step:i15}, \ref{step:i21}, and \ref{step:i24} do not return
``failure''.
We determine $a$, $b$, and $w$ in steps \ref{step:i22}--\ref{step:i10}.
In \autoref{step:i22}, we have $(b,w) \in \{(\alt{b},\alt{w}), (-\alt{b},
\alt{w} - \alt{b})\}$, depending on the choice of the order of $x_1$ and $x_2$.
 Since $f= \MultConst{\alt{a}, \alt{b}, m}^{[\alt{w}]} =
\MultConst{\alt{b}^r-{\alt{a}}, -\alt{b}, m}^{[\alt{w}-\alt{b}]}$ according
to \autoref{pro:multi_uniqueness}~\ref{item:2}, we have $f =
\MultConst{\bar{a}, b, m}^{[w]}$ for some $\bar{a} \in \{a_{0},
\alt{b}^r-{\alt{a}}\}$.  The leading coefficient of $f'$ is $-
m^{2}\bar{a}b^{1-r}(b^{r}-\bar{a})$ by \eqref{eq:60} yielding a quadratic
polynomial in $\ff[y]$ with roots $\bar{a}$ and $b^{r}-\bar{a}$ for
\autoref{step:i23}.  There, we find $A =  \{\bar{a}, b^{r}-\bar{a}\}$ and \autoref{step:i10} identifies $\bar{a}$.

For the cost over $\ff = \Fq$, the conditions in steps~\ref{step:i100} and \ref{step:i101} ensure that all powers of $\feins$ in the gcd computations of steps~\ref{step:i05} and \ref{step:i12} have degree at most $(r+2)(r-2)<n$ and we have $O(\MM(n) \log n)$ field operations for the quotients, gcds, and products in steps~\ref{step:i-1}, \ref{step:i01}, \ref{step:i05}, \ref{step:i12}, \ref{step:i14}, and \ref{step:i10}.
The $\feins$-adic expansion of $\fnull$ is a sequence $a_0, \ldots,
a_{\nu-1}  \in \Fq[x]$ such that $\fnull = \sum_{0 \leq i < \nu} a_i
\feins^i$ and $\deg a_i < \deg \feins$ for all $ i < \nu$.  We may
bound $\nu$ by the smallest power of $2$ greater than $ \deg \fnull /
\deg \feins $. Then $\nu < 2 \deg \fnull / \deg \feins$ and for $k$ in
\autoref{step:i02} we have $k+1 = \min \{ 0 \leq i < \nu \colon a_i
\neq 0\}$. We can compute the expansion with $O(\MM(\nu \deg
\feins)\log \nu)$ field operations; see \citet[Theorem
9.15]{gatger13}.  %
Thus the cost of \autoref{step:i02} is $O(\MM(n) \log n)$ field operations. The calculation of the right-hand side in \autoref{step:i10} takes $O(\MM(n) \log n)$ field operations, by first substituting $x+w$ for $x$ in  $\MultConst{a,b,m}$ as in \eqref{eq:3normal}, then computing its coefficients and leaving away the constant term.
We ignore the (cheap) operations on integers in the various tests, in \autoref{step:i17}, and the computation of derivatives in steps~\ref{step:i-1}, \ref{step:i01}, and \ref{step:i14}.
 The polynomial square root in \autoref{step:i00} and the $p^{\ell}$th root in \autoref{step:i13} take $O(n \log q)$ field operations each using $u^{q^{c}/p^{\ell}} = u^{1/p^{\ell}}$ for $u \in \Fq$ and the smallest $c \geq 1$ with $q^{c} \geq p^{\ell}$.
Taking the square roots in steps~\ref{step:i21} and \ref{step:i23} can
be done deterministically
by first reducing the computations to the
prime field $\Fp$, see \citet[Exercise 14.40]{gatger13}, %
and then
finding square roots in $\Fp$ by exhaustive search. These take $O(\log
q)$ and $O(\sqrt{n})$ field operations, respectively, since $n=r^2$ is a power of $p$.
\end{proof}

\section{Root multiplicities in collisions}
\label{sec:algebra}

In this section we describe the structure of root multiplicities in collisions over an algebraic closure of $F$ under certain conditions. In \autoref{sec:classification} these results will be used for the classification of 2-collisions at degree $p^2$.  For the classification, its proof, and the lemmas in this section, we follow ideas of \cite{dorwha74} and \cite{zan93}; an earlier version can be found in \cite{bla11}.

After some general facts about root multiplicities, we state an assumption on 2-collisions (\autoref{assum:a}) under which we determine the root multiplicities of their components (\autoref{thm:class2}).
In \autoref{ex:assumconst} we see that this assumption holds for the 2-collisions in \autoref{thm:nonadd} and in \autoref{thm:constmulti}. Then we recall the well-known relation between decompositions of polynomials and towers of rational function fields. We reformulate a result by \cite{dorwha74} about the ramification in such fields in the language of root multiplicities of polynomials (\autoref{prop:rammulti}) and derive further properties about the multiplicities in collisions for which \autoref{assum:a} holds.

We use the following notation. Let $\ff$ be a field of characteristic $p>0$ and $\kk = \overline{F}$ an algebraic closure of $\ff$.
For a nonzero polynomial $f \in F[x]$ and $b \in \kk$, let $\mult_b (f)$ denote the \emph{root multiplicity} of $b$ in $f$, so that  $f = (x - b)^{\mult_b (f)} u$ with $u \in \kk[x]$ and $u(b) \neq 0$.
For $c \in K$, we denote as $\Van{f}{c}$ the set of all $b \in
K$ such that $f(b) =c$.

\begin{lemma}
\label{lem:multmalmult}
Let $f = g \circ h \in P_n(\ff)$ and $c \in \kk$.  Then
\begin{equation}
\Van{f}{c} = \dot{\, \smashoperator{\bigcup_{a \in \Van{g}{c}}}\,} \Van{h}{a}
\end{equation}
is a partition of $\Van{f}{c}$, and for all $b \in \Van{f}{c}$, we have
\begin{equation}
\label{eq:70}
\mult_b (f-c) = \mult_{h(b)}(g-c) \cdot \mult_b (h -h(b)).
\end{equation}
\end{lemma}

The partition of $\Van{f}{c}$ from \autoref{lem:multmalmult} is illustrated in \autoref{fig:part}, where we write $\Van{g}{c} =\{a_0, a_1, a_2, \dots \}$ and $h^{-1}(a_i) = \{a_{i0},a_{i1},a_{i2},\dots\}$ for $i \geq 0$.

\begin{figure}
  \centering
\begin{tikzpicture}[
map/.style={thick, ->, >=stealth'},
scale = .7
]

\begin{scope}
\tikzstyle{every node}=[pos=0.5]

\draw (-4,6) rectangle ++(1,1) node[pos=0.5] {$b_{00}$} rectangle
++(-1,1) node {$b_{01}$} rectangle ++ (1,2) node {$b_{02}$} rectangle
++ (-1,1) node {$\vdots$};
\draw (-3,6) rectangle ++(3,1) node {$b_{10}$} rectangle ++(-3,2) node
{$b_{20}$} rectangle ++ (3,2) node {$\vdots$};
\draw (0,6) rectangle ++(3,2) node {$b_{20}$} rectangle ++(-3,3) node {$\vdots$};
\draw (3,6) rectangle ++(2,5) node {$\iddots$};

\draw[map] (-3.5,6) -- (-3.5,5) node[pos=0.5,left] {$h$};
\draw[map] (-1.5,6) -- (-1.5,5);
\draw[map] (1.5,6) -- (1.5,5);
\draw[map] (4,6) -- (4,5);

\draw (-4,4) rectangle (-3,5) node[pos=0.5] {$a_{0}$};
\draw (-3,4) rectangle (0,5) node[pos=0.5] {$a_{1}$};
\draw (0,4) rectangle (3,5) node[pos=0.5] {$a_{2}$};
\draw (3,4) rectangle (5,5) node[pos=0.5] {$\dots$};

\draw[map] (-3.5,4) -- (0.5,3) node[pos=0.5,below] {$g$};
\draw[map] (-1.5,4) -- (0.5,3);
\draw[map] (1.5,4) -- (0.5,3);
\draw[map] (4,4) -- (0.5,3);

\draw (0,2) rectangle (1,3) node[pos=0.5] {$c$};
\end{scope}

\node[rotate=45] at (-3,12) {$h^{-1}(a_{0})$};
\node[rotate=45] at (-1,12) {$h^{-1}(a_{1})$};
\node[rotate=45] at (2,12) {$h^{-1}(a_{2})$};
\node[align=center] at (-7,8.5) {$h^{-1}(g^{-1}(c))$ \\ $ = f^{-1}(c)$};
\node at (-7,4.5) {$g^{-1}(c)$};
\node at (-7,2.5) {$c \in K$};
\end{tikzpicture}
\caption{Partition of $\Van{f}{c}$} \label{fig:part}
\end{figure}

\begin{proof}
Let $b \in \bigcup_{a \in \Van{g}{c}} \Van{h}{a}$ and $a \in \Van{g}{c}$ such that $b \in \Van{h}{a}$. Hence $f(b) = g(h(b)) = g(a) = c$ and thus $b \in \Van{f}{c}$. On the other hand, let $b \in \Van{f}{c}$ and set $a = h(b)$. Then $b \in \Van{h}{a}$ and $a \in \Van{g}{c}$, since $g(a) = g(h(b)) = c$. Hence $b \in \bigcup_{a \in \Van{g}{c}} \Van{h}{a}$.  Moreover if $b \in \Van{h}{a} \cap \Van{h}{a_{0}}$ for some $a$, $a_0 \in \kk$, then $a = h(b) = a_0$.

For \eqref{eq:70}, let $b\in \Van{f}{c}$, $a = h(b)$, $e = \mult_a(g-c)$, and $\alt{e}=\mult_b (h -a)$. Then $g-c = (x - a)^e G$ and $h-a= (x -b)^{\alt{e}}H$ for some $G$, $H \in \kk[x]$ with $G(a)\cdot H(b) \neq 0$. Thus $f-c = g(h) - c = (h - a)^e G(h) = ( (x-b)^{\alt{e}} H )^e G(h) = (x-b)^{e{\alt{e}}} H^e G(h)$ with $(H^e G(h)) (b) = H(b)^e G(a) \neq 0$.
\end{proof}

\begin{lemma}
\label{lem:trivial2}
Let $f\in\kk[x]$ and $b \in \kk$.
Then $b$ is a root of $f'$ if and only if there is some $c\in \kk$ with $\mult_b (f-c) > 1$. Moreover, for any $c \in K$ with $p \nmid \mult_b (f-c)$, we have $\mult_b (f') =\mult_b (f- c) -1$.
\end{lemma}
\begin{proof}
Let $b$ be a root of $f'$ and set $c = f(b)$. Then $b$ is a root of $f - c$. We write $f - c= (x - b) u$ for some $u \in \kk[x]$. Then $f' = (f-c)'  = u + (x-b)u'$, and thus $u(b) = f'(b) =0$. Hence $b$ is a multiple root of $f-c$.

Now, let $c \in K$ with $e = \mult_b(f-c)$. Then $f-c = (x -b)^e u$
for some $u\in\kk[x]$ with $u(b) \neq 0$ and $f' = (f-c)' =
(x-b)^{e-1} ( e u + (x-b)u')$.  Thus, $b$ is a root of $f'$ if $e >
1$.  This proves the converse.  Moreover, if $p \nmid e$ then  $( e u + (x-b)u') (b) = e u(b) \neq 0$ and hence $\mult_b (f') =e -1$.
\end{proof}

We use the following proposition. The second part was stated as Proposition 6.5 (i) in \citet*{gatgie10a} for $\ff = \Fq$.

\begin{proposition}
\label{prop:2degree}
Let $r$ be a power of $p$ and $f \in P_{r^{2}}(\ff)$ have a $2$-collision
$C$ such that $\deg g = \deg h = r$ and $g' h' \neq 0$ for all $(g,h) \in C$.  Then $f' \neq 0$ and the following hold.
\begin{ronumerate}
\item \label{item:8} There are integers $d_1$ and $d_2$ such that $\deg g'  = d_1$ and $\deg h' = d_2$ for all $(g,h) \in C$.
\item \label{item:9} Furthermore, if $r = p$, then $d_1 = d_2$.
\end{ronumerate}
\end{proposition}
\begin{proof}
 \ref{item:8} Let $(g,h) \in C$ and $f = g \circ h$. Then
\begin{equation}
 \label{eq:derivdeg}
  \deg f' = \deg g' \cdot \deg h + \deg h'.
 \end{equation}
 Since $g' h' \neq 0$, this is an equation of nonnegative integers.
  Moreover, $\deg h' < \deg h = r$ and thus $\deg g'$ and $\deg h'$ are uniquely determined by $\deg f'$ and $r$, which proves the claim.

 \ref{item:9} For $r = p$, let $\ell = \deg_{2} g$ and $m = \deg_{2} h$ with the second degree $\deg_2$ as in \eqref{eq:secdeg}. Since $g' h' \neq 0$, we find $d_{1} = \deg g' = \ell - 1$ and $d_{2} = \deg h' = m - 1$ for all $(g,h)\in C$ and it is sufficient to show $\ell = m$.  We have
\begin{align}
g & = x^p + g_\ell x^\ell + O(x^{\ell -1}), \\
h & = x^p + h_m x^{m}  + O(x^{m-1})
\end{align}
with $g_{\ell}, h_{m}\in \ff^\times$. The highest terms in $h^{\ell}$ and $g \circ h$ are given by
\begin{align}
h^{\ell} & = (x^{p}+h_{m}x^{m}+O(x^{m-1}))^{\ell}  \\
& = x^{\ell p} +
\ell h_{m}x^{(\ell-1)p+m} + O(x^{(\ell-1)p+m-1}), \\
\begin{split}\label{eq:star}
g \circ h & =  x^{p^{2}} + h_{m}^{p} x^{m p} +
O(x^{(m-1)p}) + g_{\ell} x^{\ell p} + \ell g_{\ell}h_{m}x^{(\ell-1)p+m} \\
& \quad +
O(x^{(\ell-1)p+m -1}) + O(x^{(\ell-1)p}).
\end{split}
\end{align}

Algorithm 4.10 of \cite{gat12} computes
the components $g$ and $h$ from $f$, provided that $h_{p-1}\neq 0$. We
do not assume this, but can apply the same method.  Once $g_{\ell}$ and
$h_{m}$ are determined, the remaining coefficients first of $h$,
then of $g$, are computed by solving linear equations of the form
$uh_{i}=v$, where $u$ and $v$ are known at that point, and $u \neq
0$.  Quite generally, $g$ is determined by $f$ and $h$, see \autoref{lem:gunique}.

For $(g^{*},h^{*})\in C$, we find that $(g_{\ell},h_{m})=(g^{*}_{\ell}, h^{*}_{m})$
implies $(g,h)=(g^{*},h^{*})$ by the uniqueness of the procedure just sketched.
Inspection of the coefficient of
$x^{(\ell-1)p+m}$ in \eqref{eq:star} shows that $g_{\ell} = g^{*}_{\ell}$ if
and only if $h_{m}= h^{*}_{m}$.

Now take some
$(g^{*},h^{*})\in C$ and assume that $\ell \neq m$. Then $\deg_{2}(g\circ
h)$ is one of the two distinct integers $m p$ or $\ell p$.  If $m >
\ell$, then $h_{m}^{p}$ (and hence $h_{m}$) is
uniquely determined by $f$, and otherwise $g_{\ell}$ is. In
either case, we conclude from the previous observation that $(g,h) =
(g^{*},h^{*})$.  This shows $\ell= m$ if $(g,h) \neq (g^*, h^*)$.
\end{proof}

A \emph{common right component} (over $\kk$) of two polynomials $h, h^* \in \kk[x]$ is a nonlinear polynomial $v \in \kk[x]$ such that $h= u \circ v$ and $h^* = u^* \circ v$ for some $u$, $u^* \in \kk[x]$. We now state an assumption which we use in \autoref{prop:rammulti}, the lemmas thereafter, and in \autoref{thm:class2}.

\begin{assumption}
\label{assum:a}
Let $f \in P_{n}(\ff)$ have a $2$-collision $\{ (g,h),(g^*, h^*) \}$.  We consider the following conditions.
\begin{ronumerate}
\item[($A_{1}$)] \namedlabel{assum:0}{($A_{1}$)} The derivative $f'$ is nonzero.
\item[($A_{2}$)] \namedlabel{assum:-1}{($A_{2}$)} The degrees of all components are equal, that is,  $\deg g = \deg g^* = \deg h = \deg h^*$.
\item[($A_{3}$)] \namedlabel{assum:i}{($A_{3}$)} The right components $h$ and $h^*$ have no common right component over $\kk$.
\item[($A_{4}$)] \namedlabel{assum:ii}{($A_{4}$)} For all $c \in \kk$, neither $g - c$ nor $g^* - c$ have roots in $\kk$ with multiplicity divisible by $p$.
\item[($A_{5}$)] \namedlabel{assum:iii}{($A_{5}$)} The degrees of $\dif{g}$ and $\dif{h^*}$ are equal.
\end{ronumerate}
\end{assumption}

\begin{lemma}
\label{cor:ass}
Let $f \in P_{n}(\ff)$ have a $2$-collision $\{ (g,h),(g^*, h^*) \}$.
\begin{ronumerate}
\item \label{cor:ass0} Assumption \ref{assum:0} holds if and only if all derivatives $g'$, ${g^{*}}'$, $h'$, and ${h^{*}}'$ are nonzero.
\item \label{cor:ass1} If $h$ or $h^*$ is indecomposable, then \ref{assum:i} holds. In particular, it holds if $\deg h = \deg h^{*}$ is prime.
\item \label{cor:ass2} If $\deg g = p$ and \ref{assum:0} holds, then \ref{assum:ii} holds.
\item \label{cor:ass3} If $n=p^2$ and \ref{assum:0} holds, then
  \ref{assum:iii} holds.
\item \label{cor:ass4} If \ref{assum:0}, \ref{assum:-1}, and \ref{assum:iii} hold, then
  $\deg g' = \deg {g^{*}}' = \deg h' = \deg {h^{*}}'$.
\end{ronumerate}
\end{lemma}
\begin{proof}
\ref{cor:ass0} The claim follows from the fact that $f' = g' (h) \cdot
h' $.

\ref{cor:ass1} Assume that $h$ is indecomposable. Then a common right component of $h$ and $h^*$ would imply $h = h^*$ and thus $(g,h) = (g^*, h^*)$, by \autoref{lem:gunique}, a contradiction. Hence \ref{assum:i} holds. Moreover, polynomials of prime degree are indecomposable.

\ref{cor:ass2} If a root multiplicity of $g-c$ was divisible by $p$ for
some $c\in \kk$, then $g-c = (x-a)^p$ for some $a \in \kk$. This would
imply $\dif{g} = 0 $, contradicting \ref{assum:0}. Similarly, for all
$c \in \kk$ the root multiplicities of $g^*-c$ are not divisible by $p$. Thus \ref{assum:ii} holds.

\ref{cor:ass3}--\ref{cor:ass4}  We can apply
\autoref{prop:2degree}, since $\deg g = \deg g^{*} = \deg h = \deg
h^{*}$ by $n=p^{2}$ or \ref{assum:-1}, respectively, and $g'h'
{g^{*}}' {h^{*}}' \neq 0$ by \ref{assum:0} and \ref{cor:ass0}.  Then
\ref{item:9} of the cited proposition shows $\deg g'
= \deg {g^{*}}' = \deg h' = \deg {h^{*}}'$, proving  \ref{cor:ass3} and
\ref{cor:ass4}.
\end{proof}

In \autoref{ex:assumconst} we show that \autoref{assum:a} holds for the collisions in \autoref{thm:nonadd} and in \autoref{thm:constmulti}. We need the next two propositions to check \ref{assum:i} for these collisions.

\begin{proposition}
\label{prop:simpnocommon}
Let $r$ be a power of $p$, let $a, a^* \in F$ and $m$ be a positive divisor of $r-1$, $\ell = (r-1)/m$, and
\begin{align}
h &= x(x^\ell - a)^m,\\
h^* &= x(x^\ell - a^*)^m.
\end{align}
If $h$ and $h^*$ have a common right component, then $h = h^*$. In particular, the right components in $2$-collisions of the form as in \autoref{thm:nonadd} have no common right component.
\end{proposition}
\begin{proof}
By \citet[Theorem 4.1]{henmat99} it suffices to prove the claim for
additive polynomials, that is, for $m=1$. Furthermore, we can assume without loss of generality that $F$ is algebraically closed. Let $v \in
P_{p^\nu}(F)$ be a common right component of $h$ and $h^*$ with $h = u \circ v$
and $h^* = u^* \circ v$ for some $u, u^* \in P_{p^k}(F)$, $\nu \geq 1$, and $r = p^{k +\nu}$. Then $u$, $u^*$,
and $v$ are additive polynomials; see \citet[Lemma 2.4]{coh90c}. By
\citet[Theorem 3 in Chapter 1]{ore33b} and since $F$ is algebraically
closed, we may assume $\nu = 1$ and $v = x^p - bx$, for some $b \in
F$.  For $u = \sum_{0 \leq i \leq k} u_i x^{p^i}$, we have
\begin{align}
h &= x^r - ax = u \circ (x^p - bx) \\
&= u_k x^r + \sum_{1 \leq i \leq k} (u_{i-1} - u_i b^{p^i})x^{p^i} - u_0 b x.
\end{align}
Thus $u_k = 1$ and  $u_{i-1} = u_i b^{p^i} = \prod_{i \leq j\leq k} b^{p^j}$, for $1 \leq i \leq k$. Moreover, $a = u_0 b = \prod_{0 \leq j\leq k} b^{p^j}$ is uniquely determined by $b$. Thus $a = a^*$ and $h = h^*$.
\end{proof}

\begin{proposition}
\label{prop:ghindec}
Let $r$, $b$, $a$, $a^*$, $m$, $m^*$, $g$, and $h$ be as in \autoref{thm:constmulti}.
Then $g$ and $h$ are indecomposable.
\end{proposition}
\begin{proof}
Let $g = u \circ v$ with $u \in P_k(F)$, $v \in P_\ell(F)$, $k\ell = r$, and $\ell >1$. Then $p \mid \ell$.
By \autoref{lem:multmalmult} we have
\begin{equation}
\label{eq:ghindec1}
\dot{\,\smashoperator{\bigcup_{a_0 \in \Van{u}{0}}}\,} \Van{v}{a_{0}} = \Van{g}{0} = \{ 0, a\}.
\end{equation}
Since $v$ is original, we have $\{0\} \subseteq \Van{v}{0} \subseteq \{ 0, a\}$. If $\Van{v}{0} = \{ 0\}$, then $v = x^\ell$ and thus $p \mid \ell \mid m$, by \eqref{eq:70}, in contradiction to $p \nmid m$. Thus $\Van{v}{0} = \{ 0, a\}$. Since the union in \eqref{eq:ghindec1} is disjoint, we find that $\Van{u}{0} = \{0\}$ and $0$ is the only root of $u$.
Hence $u = x^k$ and $k \mid \gcd(m, m^*) = 1 $,
by \eqref{eq:70} and \autoref{pro:multi_uniqueness}~\ref{item:-1}. Therefore $u$ is linear and thus $g$ is indecomposable.

By \eqref{eq:48} and \autoref{pro:multi_uniqueness}~\ref{item:-1}, we find $h = x^{m^*} H$ and $h-a = (x-b)^m \qepol^{*}$ with squarefree polynomials $H$ and $H^*$. Thus $h^{[b]} = x^m \tilde{H}$ for squarefree $\tilde{H} = H^* \circ (x + b)$. We find that $h$ is decomposable if and only if $h^{[b]}$ is decomposable. By \autoref{pro:multi_uniqueness}~\ref{item:0} either $m > r/2$ or $m^* > r/2$. If $m > r/2$, then we rename $h$ as $h^{[b]}$, $H$ as $\tilde{H}$ and $m$ as $m^*$. We have in either case $m^* > r/2$.

Now let $h = u \circ v$ with $u \in P_k(F)$, $v \in P_\ell(F)$, $k\ell = r$, and $\ell >1$. Then $p \mid \ell$. The only multiple root in $h$ is $0$, since $H$ is squarefree, by \autoref{pro:multi_uniqueness}~\ref{item:-1}. Its multiplicity is $\mult_0(h) = m^* = \mult_0 (u) \cdot \mult_0(v)$. Thus $\mult_0(v) \mid m^*$ and hence $p \nmid \mult_0(v)$. Since the multiplicities of $v$ sum up to $\ell$, which is divisible by $p$,
 there is another root $b_0 \neq 0$ of $v$.
  Then $1 = \mult_{b_0} (h) = \mult_0 (u) \cdot \mult_{b_0} (v)$ and thus $\mult_0(u) = 1$.
Hence  $\mult_0 (v) = m^*$. We have $\ell > m^* > r/2$, thus $\ell = r$ and $u$ is linear.
\end{proof}

\begin{example}
\label{ex:assumconst}
\autoref{assum:a} holds for the $\#T$-collisions in
\autoref{thm:nonadd} with $\# T \geq 2$ and the 2-collisions in
\autoref{thm:constmulti}. In both cases \ref{assum:-1} holds by definition. Assumption \ref{assum:i} follows from \autoref{prop:simpnocommon} and from \autoref{prop:ghindec} and \autoref{cor:ass}~\ref{cor:ass1}, respectively.

The derivatives of the components in \autoref{thm:nonadd} are
\begin{equation}
\label{eq:simpderiv}
\begin{split}
g' &= - u s^r t^{-1} (x^\ell - u s^r t^{-1})^{m-1}, \\
h' &= - st (x^\ell - st)^{m-1}.
\end{split}
\end{equation}
Since $u$, $s$, $t \in \ff^\times$, we find $\deg g' = \deg h' = \ell(m-1) \geq 0$, independent of $t \in T$, and thus \ref{assum:iii} holds. By \eqref{eq:derivdeg}, $\deg f' \geq 0$ and thus \ref{assum:0} holds.
If there is $c \in K$ such that $g-c$ has a multiple root $b\in K$,
then $b$ is also a root of $g'$ by \autoref{lem:trivial2}.  Since
$\Van{g'}{0} \subseteq \Van{g}{0}$ by \eqref{eq:simpderiv}, we have only simple roots in $g-c$ for $c \neq
0$.  The multiple roots of $g$ have multiplicity $m$ and
\ref{assum:ii} follows from $p \nmid m \mid r-1$.

For the collisions in \autoref{thm:constmulti}, \ref{assum:ii} follows
similarly from $p \nmid mm^*$.  Finally, \ref{assum:0} and
\ref{assum:iii} are satisfied by \eqref{eq:60} and $a$, $a^*$, $b \in \ff^\times$.
\end{example}

\begin{lemma}
\label{lem:trivial}
Let $f \in \ff[x]$ be monic and $y$ be transcendental over $\kk(x)$. Then $f - y \in \kk(y)[x]$ is irreducible.
\end{lemma}
\begin{proof}
Assume $f -y = uv$ for some $u, v \in \kk[x,y]$. The degree in $y$ of $f-y$ is  $\deg_y(f-y) = 1 = \deg_y u + \deg_y v$. Thus we may assume $\deg_y u = 1$ and $\deg_y v = 0$. Then $ a v = -1$, where $a \in \kk[x]$ is the leading coefficient of $u$ in $y$. Thus $v \in \kk[x]^\times = \kk^\times$ and $f-y$ is irreducible in $\kk[x,y]$. A factorization of $f-y$ in $\kk(y)[x]$ yields a factorization in $\kk[x,y]$, by the Lemma of Gau\ss, see \citet[Corollary 2.2 in Capter IV]{lan02}. Hence $f-y$ is also irreducible in $\kk(y)[x]$.
\end{proof}

 Let $f \in P_n(\ff)$ with $\dif{f} \neq 0$ and $y$ be transcendental
 over $\kk(x)$. Then $f-y \in \kk(y)[x]$ is irreducible and separable
 over $\kk(y)$, by \autoref{lem:trivial} and since the derivative of
 $f-y$ with respect to $x$ is $(f-y)' = f' \neq 0$.  In particular,
 $f-y \in F(y)[x]$ is irreducible and separable.  Let $\alpha \in \overline{\kk(y)}$ be a root of $f - y$. Then $\kk(y)[\alpha] = \kk(\alpha)$ is a rational extension of $\kk(y)$ of degree $n$.
Let $\mathcal{M}$ be the set of intermediate fields between $\kk(\alpha)$ and $\kk(y)$ and $\mathcal{R} = \{ h \in P_m(\kk) \colon m \mid n \text{ and there is } g \in P_{n/m}(\kk) \text{ such that } f = g \circ h\}$ be the set of right components of $f$.

\begin{fact}[\cite{frimac69}, Proposition 3.4]
\label{thm:bij}
Let $f \in P_n(\kk)$ with $\dif{f} \neq 0$ and let $\alpha \in \overline{\kk(y)}$ be a root of $f - y \in \kk(y)[x]$.
Then the map
\begin{equation}
\label{eq:bij}
\begin{split}
\mathcal{R} &\rightarrow \mathcal{M}, \\
h &\mapsto \kk(h(\alpha))
\end{split}
\end{equation}
is bijective.
\end{fact}

The fact follows from \citet[Proposition 3.4]{frimac69}.
Indeed, for each $u \in \kk[x]$ of degree $m$ there is exactly one $v \in P_m(\kk)$ such that $u = \ell \circ v$ for some linear polynomial $\ell \in \kk[x]$; see \citet[Section 2]{gat12}.

 The sets $\mathcal{R}$ and $\mathcal{M}$ can be equipped with natural lattice structures for which \eqref{eq:bij} is an isomorphism.

We now use the theory of places and ramification indices in function fields; see \cite{sti09} for the background. A \emph{place} in a function field $L$ over $K$ is the maximal ideal of some valuation ring of $L$ over $K$. For an finite extension $M$ of $L$ a place $\prip$ in $M$ is said to lie over a place $P$ in $L$ if $P \subseteq \prip$. Then we write $\prip \mid P$ and define the ramification index of $\prip \mid P$ as the integer $e(\prip \mid P)$ such that $v_\prip (a) = e(\prip \mid P) \cdot v_P (a)$ for all $a \in L$, where $v_\prip$ and $v_P$ are the corresponding valuations of $\prip$ and $P$, respectively; see \citet[Proposition 3.1.4 and Definition 3.1.5]{sti09}.

Later, we translate this into the language of root multiplicities of
polynomials.  First, we need the following result, which is proven in \citet[Lemma 1]{dorwha74} for rational function fields under the assumption that the characteristic of $K$ is zero. Our proof avoids this assumption.

\begin{theorem}
\label{lem:Dorey}
Let $L$, $M$, $M^*$, $N$ be function fields over $K$ such that $L \subseteq M, {M^*} \subseteq N$ are finite separable field extensions and $M \otimes_{L} {M^*} \cong M M^* = N$. Let $P$ be a place in $L$, and $\prip$, $\priq$ be places over $P$ in $M$ and ${M^*}$, respectively. Assume that at least one of the ramification indices $m = e(\prip \mid P)$ and $m^* = e(\priq \mid P)$ is not divisible by the characteristic of $\kk$. Then there are $\gcd(m, m^*)$ places $\priP$ in $N$ which lie over $\prip$ and over $\priq$. Moreover, for such a place we have  $e(\priP \mid P) = \lcm(m,m^*)$.
\end{theorem}

\begin{proof}
Abhyankar's Lemma says that for a place $\priP$ in $N$ over $\prip$ and over $\priq$,
\begin{equation}
\label{eq:ramindexlcm}
e(\priP \mid P) = \lcm(m, m^*),
\end{equation}
see \citet[Theorem 3.9.1]{sti09}.
Now we proceed as in \cite{dorwha74}. For places $\prip$, $\priq$, and $\priP$ over $P$ in $M$, $M^*$, and $N$, respectively, we denote by $\Lambda = \widehat{L}$, $\widehat{M}^\prip$, $\widehat{M^*}{}^\priq$, and $\widehat{N}^\priP$ the completions of $L$, $M$, $M^*$, and $N$ with respect to $P$, $\prip$, $\priq$, and $\priP$, respectively.
The tensor product  $N \otimes_M \widehat{M}^\prip$ is the direct sum
of the completions of $N$ with respect to the places in $N$ over
$\prip$, and $M^* \otimes_{L} \Lambda$ is the direct sum of the completions of $M^*$ with respect to the places  in $M^*$ over $P$; see \citet[Proposition~8.3 in Chapter II]{neu99}.
Since $M \otimes_{L} M^* \cong N$, we have
\begin{align}
\label{eq:tensor}
\bigoplus_{\priP \mid \prip} \widehat{N}^\priP &\cong N \otimes_{M} \widehat{M}^\prip \cong M^* \otimes_{L} M  \otimes_{M} \widehat{M}^\prip \cong M^*  \otimes_{L} \widehat{M}^\prip \\
&\cong M^*  \otimes_{L} (\Lambda \otimes_{\Lambda} \widehat{M}^\prip) \cong (M^*  \otimes_{L} \Lambda )\otimes_{\Lambda} \widehat{M}^\prip \\
&\cong \bigoplus_{\prip_0^* \mid P} \widehat{M^*}{}^{\prip_0^*} \otimes_{\Lambda} \widehat{M}^\prip,
\end{align}
where the last direct sum is taken over all places $\prip_0^*$ in $M^*$ over $P$.
We show that $\widehat{M^*}{}^\priq \otimes_{\Lambda}
\widehat{M}^\prip$ is the direct sum of the completions of $N$ with respect to the places that lie over $\prip$ and $\priq$. For this purpose, consider the (external) composite fields of $\widehat{M^*}{}^\priq$ and $\widehat{M}^\prip$ in an algebraic closure $\Omega$ of $\widehat{N}^\priP$; those are the field extensions $\Gamma \subseteq \Omega$ of $\Lambda$ such that there are two field homomorphisms which map $\widehat{M^*}{}^\priq$ and $\widehat{M}^\prip$, respectively, into $\Gamma$ and whose images generate $\Gamma$. Then $\widehat{M^*}{}^\priq \otimes_{\Lambda} \widehat{M}^\prip$ is the direct sum of the composite fields of $\widehat{M^*}{}^\priq$ and $\widehat{M}^\prip$; see \citet[Theorem 21 in Chapter I]{jac64}. Each such composite field $\Gamma$ is isomorphic to a summand in $\bigoplus_{\priP \mid \prip} \widehat{N}^\priP$, by the Krull-Remak-Schmidt Theorem; see \citet[Theorem 7.5]{lan02}. Thus there exists $\priP \mid \prip$ such that $\Gamma = \widehat{N}^\priP$. Since $\Gamma$ is an extension of $\widehat{M^*}{}^{\priq}$, we find $\priP \mid \priq$ as claimed. On the other hand, for a place $\priP$ in $N$ over $\prip$ and $\priq$, $\widehat{N}^\priP$ is a composite field of $\widehat{M^*}{}^\priq$ and $\widehat{M}^\prip$ and thus is a summand in $\widehat{M^*}{}^\priq \otimes_{\Lambda} \widehat{M}^\prip$.

The summands of $\widehat{M^*}{}^{\priq} \otimes_{\Lambda} \widehat{M}^\prip$ are of degree $\lcm(m, m^*)$, by \eqref{eq:ramindexlcm}, and the $\Lambda$-dimension of $\widehat{M^*}{}^\priq \otimes_{\Lambda} \widehat{M}^\prip$ is $m m^*$. Thus there are $m m^* / \lcm(m, m^*) = \gcd(m,m^*)$ places over $\prip$ and $\priq$.
\end{proof}

In the following we link the notion of places and ramification indices to the notion of roots and root multiplicities.
Let $K(t)$ be a rational function field.
Then the local ring $\Oring_\infty = \{g/h \in K(t) \colon g,h \in K[t], \deg g \leq \deg h\}$ is the $1/t$-adic valuation ring of $K(t)$ and $P_\infty = (1/t) \Oring_\infty$ is its maximal ideal.
For $c\in K$, the local ring $\Oring_{t-c} = \{g/h \in K(t) \colon g,h \in K[t], h(c) \neq 0 \}$ is the $(t-c)$-adic valuation ring of $K(t)$ and $P_c = (t-c) \Oring_{t-c}$ is its maximal ideal. We denote the $(t-c)$-adic valuation by $v_{P_c}$. Then we have for $f\in K[x]$
\begin{equation}
\label{eq:valmult}
v_{P_c} (f (t) ) = \mult_c(f).
\end{equation}
Since the irreducible polynomials in $K[t]$ are linear, the places $P_\infty$ and $P_c$ for all $c \in K$ are pairwise distinct and comprise all places in $K(t)$; see \citet[Theorem 1.2.2]{sti09}.
We call the places $P_c$ \emph{finite} places. The map
\begin{equation}
\label{eq:placeroot1}
\begin{split}
K &\rightarrow \{ P \colon P \text{ is a finite place in }K(t)\}, \\
c &\mapsto P_c
\end{split}
\end{equation}
is bijective.

\begin{lemma}
\label{fact:placeroot}
Let $f \in P_n(\kk)$ with $\dif{f} \neq 0$, let $\alpha \in
\overline{\kk(y)}$ be a root of $f - y \in \kk(y)[x]$, let $b,c \in K$, and let $P_c$ and $\priP_b$ be the corresponding finite places in $K(y)$ and $K(\alpha)$, respectively.
Then $\priP_b \mid P_c$ if and only if $f(b)=c$. Furthermore
\begin{equation}
\label{eq:placeroot3}
e(\priP_b \mid P_c) = \mult_b (f-c).
\end{equation}
\end{lemma}
\begin{proof}
Let $\priP_b \mid P_c$. Then $y-c \in \priP_b$ and thus $f(\alpha) -c = y-c  = (\alpha - b) g / h$ for $g, h \in K[\alpha]$ with $h(b) \neq 0$. Hence, $f(b) - c = (b -b) g(b) / h(b) = 0$.

Conversely, let $f(b) = c$. Then $\alpha - b \mid f(\alpha) - c$ in $K[\alpha]$ . Let $(y-c) g/h \in P_c$ for some $g,h \in K[y]$ with $h(c) \neq 0$. Then $h(f(b)) = h(c) \neq 0$ and thus $(y-c) g/h = (f(\alpha)-c) g(f(\alpha))/h(f(\alpha)) \in \priP_b$.

By \eqref{eq:valmult} and since $v_{P_c} (y-c) = 1$, we have $e(\priP_b \mid P_c) = v_{\priP_b} (y-c) = v_{\priP_b} (f(\alpha)-c) = \mult_b (f -c)$.
\end{proof}

\begin{proposition}
\label{prop:rammulti}
 Let $c \in \kk$ and $f \in P_{n}(\ff)$ have a $2$-collision $\{ (g,h),(g^*, h^*) \}$ satisfying \ref{assum:0}--\ref{assum:ii} in \autoref{assum:a}.
For $a \in \Van{g}{c}$ and $a^* \in \Van{{g^{*}}}{c}$, there are exactly $\gcd \left( \mult_{a} (g - c), \mult_{a^*}(g^* - c) \right)$ roots $b \in \Van{f}{c}$ such that $h(b) = a$ and $h^*(b) = a^*$. Furthermore, for each such root $b$ we have
  \begin{equation}
  \label{eq:50}
\mult_b(f - c) = \lcm \left( \mult_a (g - c), \mult_{a^*}(g^* - c) \right).
  \end{equation}
\end{proposition}
\begin{proof}
By \ref{assum:0} we have $f' \neq 0$ and thus $f-y \in F(y)[x]$ is irreducible and separable; see \autoref{lem:trivial} and the paragraph thereafter. Let $\alpha \in \overline{\kk(y)}$ be a root of $f-y$, $M = \kk (h(\alpha) ) $ and $M^* = \kk (h^*(\alpha))$, as in \eqref{eq:bij}.
Then $\alpha$ is a root of $h - h(\alpha)$ and by \autoref{lem:trivial}, $h - h(\alpha)$ is irreducible in $M[x]$. Thus the minimal polynomial of $\alpha$ over $M$ is $h - h(\alpha)$, and similarly the minimal polynomial of $h(\alpha)$ over $\kk(y)$ is $g - y$. Hence  $[\kk(\alpha) \colon M ] = \deg h$ and $[M \colon \kk(y) ] = \deg g$.
\autoref{fig:fields} illustrates the relation between these field extensions and their respective minimal polynomials.

\begin{figure}[h!]
\centering
\begin{tikzpicture}
\begin{scope}[
xscale =4.5,
yscale =3
]
    \tikzstyle{every node} = [rectangle]
    \node (M) at (-.5,0) {$M=\kk(h(\alpha))$}
     node (M*) at (.5,0) {${M^*} = \kk(h^{*}(\alpha))$}
     node (Ka) at (0,1) {$\kk(\alpha)$}
     node (Ky) at (0,-1) {$\kk(y)$};
    \draw (M) -- (Ka) node[pos=.5,left] {$h-h(\alpha)$};
    \draw (M*) -- (Ka) node[pos=.5,right] {$h^{*}-h^{*}(\alpha)$};
    \draw (Ky) -- (M) node[pos=.5,left] {$g-y$};
    \draw (Ky) -- (M*) node[pos=.5,right] {$g^{*}-y$};
\end{scope}
\end{tikzpicture}
\caption{Lattice of subfields}
\label{fig:fields}
\end{figure}
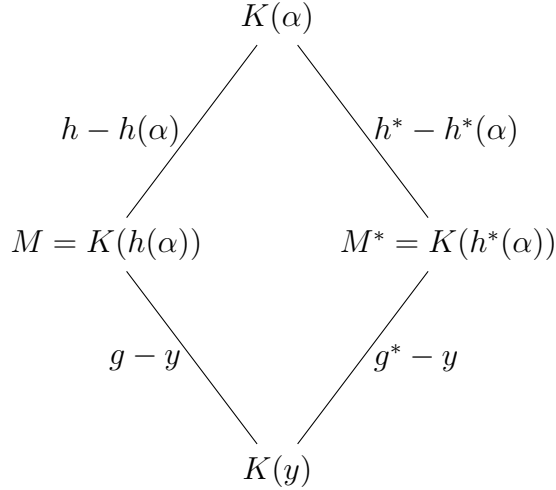

By \autoref{thm:bij} and since $M {M^*} \subseteq K(\alpha)$, there is a monic original $v \in \kk[x]$ such that $M {M^*} = \kk(v(\alpha))$. Since $M \subseteq M {M^*}$, there is $u \in \kk[x]$ such that $h= u \circ v$, by applying \autoref{thm:bij} to $K(\alpha) \mid M$. Similarly, there is $u^* \in \kk[x]$ such that $h^* = u^* \circ v$. Hence $v = x$, by \ref{assum:i}, and $M {M^*} = \kk (\alpha)$. Moreover, $M M^*$ is contained in $M \otimes_{K(y)} M^*$ as a direct summand; see \citet[Theorem 21 in Chapter I]{jac64}. Their $K(y)$-dimensions both equal $\deg f = \deg g \cdot \deg h = (\deg g)^2$, by \ref{assum:-1}. Thus $M\otimes_{K(y)} M^* \cong M M^* = K(\alpha)$.
Let $P_c$ be as in \eqref{eq:placeroot1}. Since, by \autoref{fact:placeroot}, the root multiplicities of $g - c$ are the ramification indices of the places over $P_c$ in $M$, \ref{assum:ii} rules out finite wildly ramified places in $M \mid \kk (y)$.
Thus we can apply \autoref{lem:Dorey}, as follows.

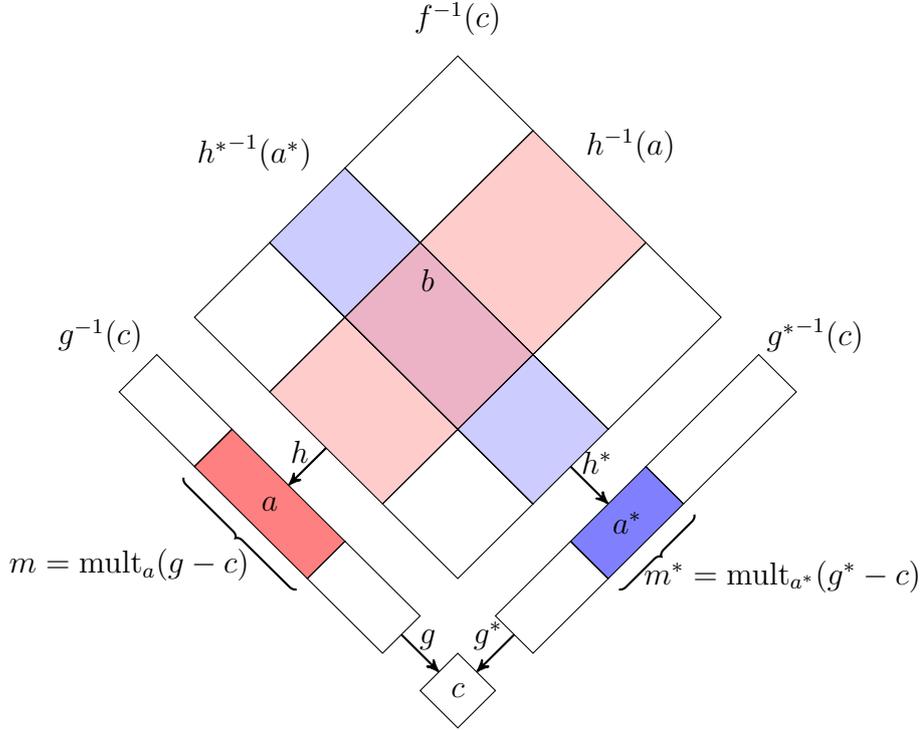
\begin{figure}[h!]
\begin{tikzpicture}
\begin{scope}[
rotate=45,
scale=0.7,
map/.style={thick, ->, >=stealth'},
]

\draw (0,0) rectangle (1,1) node[pos=0.5] {$c$};

\fill[fill=red!50] (1,4) rectangle ++ (-1,3);
\fill[fill=red!20] (2,4) rectangle ++ (7,3);
\draw (0,2) rectangle ++(1,2) rectangle ++(-1,3) node[pos=0.5] {$a$}
rectangle ++(1,2);

\fill[fill=blue!50] (4,1) rectangle ++(2,-1);
\fill[fill=blue!20] (4,2) rectangle ++(2,7);
\fill[fill=purple!30] (4,4) rectangle ++(2,3);
\draw (2,0) rectangle ++(2,1) rectangle ++(2,-1) node[pos=0.5]
{$a^{*}$} rectangle ++(3,1);

\draw (2,2) rectangle ++(2,2) rectangle ++(2,-2) rectangle ++(3,2);
\draw (2,7) rectangle ++(2,-3) rectangle ++(2,3)
node[pos=0.8] {$b$} rectangle ++(3,-3);
\draw (2,7) rectangle ++(2,2) rectangle ++(2,-2) rectangle ++(3,2);

\draw[map] (0.5,2) -- (0.5,1) node[pos=0.7,above] {$g$};
\node at (0.5,10) {$g^{-1}(c)$};
\draw[map] (2,0.5) -- (1,0.5) node[pos=0.7,above] {$g^{*}$};
\node at (10,0.5) {${g^{*}}^{-1}(c)$};

\draw[thick,decorate,decoration={brace,raise=0.2cm}] (0,4) -- ++ (0,3)
node[pos=0.5,xshift=-1.6cm,below=0.2cm] {$m = \mult_{a}(g-c)$};
\draw[map] (2,5.5) -- (1,5.5) node[pos=0.7,above] {$h$};
\node at (10.1,5.5) {$h^{-1}(a)$};

\draw[map] (5,2) -- (5,1) node[pos=0.7,right=0.1cm,above=0.1cm] {$h^{*}$};
\draw[thick,decorate,decoration={brace,raise=0.2cm}] (6,0) -- ++
(-2,0) node[pos=0.5,xshift=1.8cm,below=0.1cm] {$m^{*} = \mult_{a^{*}}(g^{*}-c)$};
\node at (5,10.4) {${h^{*}}^{-1}(a^{*})$};

\node at (9.5,9.5) {$f^{-1}(c)$};
\end{scope}
\end{tikzpicture}
\caption{Roots and multiplicities}
\label{fig:multiplicities}
\end{figure}

Let $m = \mult_a (g-c)$ and $m^* = \mult_{a^*} (g^*-c)$, see \autoref{fig:multiplicities}. By \autoref{fact:placeroot}, there are finite places $\prip_a$ and $\prip_{a^*}^*$ over $P_c$ in $M$ and $M^*$, respectively, with $m = e(\prip_a \mid P_c) $ and $m^*= e(\prip_{a^*}^* \mid P_c)$. Then, by \autoref{lem:Dorey}, there are $\gcd(m, m^*)$ places $\priP$ over $\prip_a$ and $\prip_{a^*}^*$ in $K(\alpha)$.
By the bijection \eqref{eq:placeroot1}, for each such place $\priP$ there is $b \in K$ such that $\priP = \priP_b$, and by applying \autoref{fact:placeroot} to $K(\alpha) \mid M$ and to $K(\alpha) \mid M^*$, we find $b \in \Van{h}{a} \cap \Van{{h^{*}}}{a^{*}} \subseteq \Van{f}{c}$.
On the other hand, for $b \in \Van{h}{a} \cap \Van{{h^{*}}}{a^{*}}$, the place $\priP_b$ lies over $\prip_a$ and $\prip_{a^*}^*$. Thus $\# \Van{h}{a} \cap \Van{{h^{*}}}{a^{*}} = \gcd(m,m^*)$ and $\mult_b (f-c) = e(\priP_b \mid P_c) = \lcm (m, m^*)$, by \autoref{lem:Dorey}.
\end{proof}

Combining \eqref{eq:50} and \eqref{eq:70}, for $b\in K$, $a= h(b)$, $a^* = h^* (b)$, and $c = f(b)$, we find
$\mult_{a}(g-c) \cdot \mult_b (h -a) = \mult_b (f-c) = \lcm \left( \mult_a (g - c), \mult_{a^*}(g^* - c) \right)$
and thus
 \begin{equation}
\label{eq:5070}
 \mult_b (h -a) = \lcm \left( \mult_a (g - c), \mult_{a^*}(g^* - c) \right) / \mult_a (g - c).
\end{equation}
Hence, the root multiplicities of $h-a$ are determined by those of $g -c$ and $g^*-c$.

From \autoref{prop:rammulti} we derive further results about the root multiplicities of  $f$, $g$, and $g^*$.
\begin{lemma}
\label{lem:c}
Let $c \in \kk$, $r$ be a power of $p$, $f \in P_{r^2}(\ff)$ have a $2$-collision $\{ (g,h),(g^*, h^*) \}$ satisfying \autoref{assum:a}, and let $a\in \Van{g}{c}$ and $e = \mult_a(g-c)$. Then the following hold.
\begin{ronumerate}
\item\label{item:c1} We have
  \begin{equation}
    \label{eq:49}
\gcd \{ \mult_{a^*}(g^*-c) \colon a^*
  \in \Van{{g^{*}}}{c} \} = 1.
  \end{equation}
In particular, if $e$ divides $\mult_{a^{*}} (g^{*} - c)$ for all roots $a^{*} \in \Van{{g^{*}}}{c}$, then $e = 1$.
\item\label{item:c2} The multiplicity $e$ either equals $1$ or divides $\mult_{a^*}(g^*-c)$ for all roots $a^* \in \Van{{g^{*}}}{c}$ but exactly one.
\end{ronumerate}
\end{lemma}
\begin{proof}
 \ref{item:c1} Let $d$
 be the $\gcd$ of all root multiplicities of $g^* - c$. Then $d$ divides
 $\sum_{a^* \in \Van{{g^{*}}}{c}} \mult_{a^*}(g^*-c) = \deg(g^*-c) = r$. Thus
 $d$ is a power of $p$ and hence all multiplicities of $g^*-c$ are
 divisible by $p$ if $d >1$, which contradicts \ref{assum:ii}, and
 \ref{item:c1} follows.

Before we start with the proof of \ref{item:c2}, we introduce some
notation and results for arbitrary $c \in K$, $a \in g^{-1}(c)$, and
$a^{*}\in {g^{*}}^{-1}(c)$.  We define
\begin{equation}
\label{eq:iandj}
\begin{split}
\ii (c,g) &= \smashoperator{\sum_{a\in \Van{g}{c}}} \mult_a(g'), \\
\ii (c,h^*) &= \smashoperator{\sum_{b\in \Van{f}{c}}} \mult_b({h^*}'), \\
\jj (a,a^{*}) & = \smashoperator{\sum_{b \in \Van{h}{a} \cap \Van{{h^{*}}}{a^{*}}}} \mult_b(\dif{h^*}) ,
\end{split}
\end{equation}
and have
\begin{equation}
\label{eq:anothersum}
\begin{split}
\sum_{c \in \kk} \ii (c,g) & = \deg g',\\
\sum_{c \in \kk} \ii (c,h^*) & = \deg {h^*}', \\
\smashoperator{\sum_{\substack{a \in \Van{g}{c} \\
 		a^* \in \Van{{g^{*}}}{c}}}}
 		\jj(a, a^*) & = \ii(c,h^{*}),
\end{split}
\end{equation}
since $\dot{\bigcup}_{c \in \kk} \Van{g}{c} = \kk$, $\dot{\bigcup}_{c
  \in \kk} \Van{f}{c} = \kk$, and
\begin{equation}
  \label{eq:27}
\Van{f}{c}= \dot{\, \smashoperator{\bigcup_{\substack{a \in \Van{g}{c} \\ a^* \in
    \Van{{g^{*}}}{c}}}} \,} \Van{h}{a}\,\cap\, \Van{{h^{*}}}{a^{*}}
\end{equation}
by \autoref{lem:multmalmult}.

By \ref{assum:ii}, $p \nmid \mult_a(g-c)$ and thus $\mult_a(g') = \mult_a(g-c) - 1$, by \autoref{lem:trivial2}.
Hence for $c \in \kk$ we have
\begin{equation}
\label{eq:sumdeg}
\ii (c,g) = \smashoperator{\sum_{a \in \Van{g}{c}}} ( \mult_a(g - c) - 1 ) = \deg g - \# \Van{g}{c}.
\end{equation}

Let $e=\mult_a(g - c)$ and $\et=\mult_{a^*}(g^* - c)$.
By \autoref{prop:rammulti}, the set $\Van{h}{a} \cap
\Van{{h^{*}}}{a^{*}}$ has size $\gcd(e,e^{*})$ and for a root $b \in
\Van{h}{a} \cap \Van{{h^{*}}}{a^{*}}$, we have $\mult_b(h^*-a^*) =
\mult_b(f-c) / \et = \lcm(e, \et ) / \et$, by \eqref{eq:5070}. Thus
$\mult_{b}({h^{*}}') = \lcm(e, e^{*})/e^{*} - 1$ by \ref{assum:ii} and
\autoref{lem:trivial2} and we have
\begin{equation}
  \label{eq:sumdiedum}
j(a,a^{*}) = \gcd(e, \et) \cdot (\lcm(e, \et)/ \et - 1) = e - \gcd(e, \et).
\end{equation}

We now show
\begin{equation}
  \label{eq:29}
  \smashoperator{\sum_{a^{*} \in {g^{*}}^{-1}(c)}} j(a,a^{*}) \geq e-1.
\end{equation}
Let $a_{0}^{*}, \dots, a_{\ell}^{*}$ be the roots of $g^{*}-c$ in $K$
and $e_{i}^{*} = \mult_{a^{*}_{i}}(g^{*}-c)$ be their multiplicities.
If $e$ divides all
$e_{i}^{*}$, then $e = 1$ by \ref{item:c1} and \eqref{eq:29} follows
trivially.  If $e$ divides all $e_i^*$ except exactly one, say $e
\nmid e_0^*$ and $e \mid e_i^*$ for $1 \leq i \leq \ell$, then  the
$\gcd$ of $e$ and $\et_0$ divides all $\et_i$ and hence divides
$\gcd\{ \et_i \colon 0 \leq i \leq  \ell\} =1$; see
\eqref{eq:49}. Thus $\gcd(e, \et_0)=1$, $\jj (a, a^*_0) = e-1$ by
\eqref{eq:sumdiedum}, and \eqref{eq:29} follows.

Now assume that $e$
does not divide at least two $e_{i}^{*}$, say $e \nmid
e_0^{*}$ and $e\nmid e_1^{*}$.  Then $\gcd(e, e_i^*) \neq e$, $\gcd
(e,e_{i}^{*}) \leq e/2$, and $\jj (a,a^*_i) \geq e/2$  by
\eqref{eq:sumdiedum} for $i = 0,1$.  Hence, \eqref{eq:29} holds with
strict inequality.  Summing both sides of \eqref{eq:29} over all roots of $g-c$ yields
\begin{equation}
\label{eq:asum}
 \ii (c,h^*) = \smashoperator{\sum_{\substack{a \in \Van{g}{c} \\
 		a^* \in \Van{{g^{*}}}{c}}}}
 		\jj(a, a^*) > \smashoperator{\sum_{a \in \Van{g}{c}}} (\mult_{a}(g-c)-1) = \ii (c,g)
 \end{equation}
by \eqref{eq:iandj} and \eqref{eq:sumdeg}.  With
\eqref{eq:anothersum}, this leads to
\begin{equation}
  \label{eq:26}
  \deg {h^{*}}' > \deg g',
\end{equation}
a contradiction to \ref{assum:iii}.
\end{proof}

\begin{lemma}
\label{lem:squareful}
Let $c \in \kk$, $r$ be a power of $p$, and let $f \in P_{r^2}(\ff)$ have a $2$-collision $\{ (g,h),(g^*, h^*) \}$ satisfying \autoref{assum:a}. Then the following statements are equivalent.
\begin{ronumerate}
\item\label{item:square1} $g-c$ is squareful.
\item\label{item:square2} $g^*-c$ is squareful.
\item\label{item:square3} $f-c$ is squareful.
\end{ronumerate}
Furthermore, if $g-c$ is squareful, then $g - c$ has at most one simple root.
\end{lemma}
\begin{proof}
 Assume that $g-c$ is squareful. Then there is a root of $g-c$ with
 multiplicity greater than $1$. This multiplicity divides all
 multiplicities of $g^* -c$ but exactly one, by
 \autoref{lem:c}~\ref{item:c2}. Hence all multiplicities of $g^* -c$ but at most one are greater than $1$. Thus $g^*-c$ is squareful and has at most one simple root.
We interchange the r\^oles of $g$ and $g^*$ in \autoref{lem:c} and obtain the equivalence of \ref{item:square1} and \ref{item:square2} and the last claim.

Now let $a\in\kk$ be a multiple root of $g-c$, and $b \in \Van{h}{a}$. Then $\mult_b(f-c) = \mult_{a}(g-c) \cdot \mult_b(h -h(b)) > 1$, by \autoref{lem:multmalmult}, and thus $f-c$ is squareful.

It is left to prove that if $f-c$ is squareful, then $g-c$ or $g^*-c$ is squareful. Let $b\in \kk$ be a multiple root of $f -c$. Then $1 < \mult_b(f-c) = \lcm (\mult_{h(b)} (g-c), \mult_{h^*(b)} (g^* -c) )$, by \autoref{prop:rammulti}. Thus $\mult_{h(b)} (g-c) > 1$ or $\mult_{h^*(b)} (g^* -c) > 1$.
 \end{proof}

\begin{lemma}
\label{lem:oneram}
Let $r$ be a power of $p$, and let $f \in P_{r^2}(\ff)$ have a $2$-collision $\{ (g,h),(g^*, h^*) \}$ satisfying \autoref{assum:a}. Then the following hold.
\begin{ronumerate}
\item\label{item:oneram3} There is at most one $c\in \kk$ such that $f-c$ is squareful.
\item\label{item:samenumber} For all $c\in \kk$, $\# \Van{g}{c} = \# \Van{{g^{*}}}{c}$.
\end{ronumerate}
\end{lemma}
\begin{proof}
\ref{item:oneram3}  Assume $g-c$ is squareful, for some $c \in
K$. Then $g-c$ has at most one simple root, by
\autoref{lem:squareful}. Thus $r = \deg g = \sum_{a\in \Van{g}{c}}
\mult_a(g-c) \geq 1+2(\# \Van{g}{c} -1)$. Hence $\# \Van{g}{c} \leq
(r+1)/2$ and thus $\ii (c,g) = r - \# \Van{g}{c} \geq (r-1)/2$, by
\eqref{eq:sumdeg}.  Now, if there is another value
$c_{0} \in \kk \setminus \{c\}$ such that $g-c_{0}$ is also squareful, then $r-2 \geq \deg \dif{g} =
\sum_{c\in \kk} \ii (c,g) \geq r-1$, by \eqref{eq:anothersum}. By this
contradiction, there is at most one $c$ in $\kk$ such that $g-c$ is squareful. Hence there is at most one $c$ in $\kk$ such that $f-c$ is squareful, by \autoref{lem:squareful}.

\ref{item:samenumber}  If $g-c$ is squarefree, then so is $g^{*}-c$,
by \autoref{lem:squareful}, and both have exactly $\deg g = \deg g^{*}
= r$ roots.  If $g-c$ is squareful, then by \ref{item:oneram3},
$c$ is unique with this property and thus the roots of $g'$ are the
multiple roots of $g-c$ by \autoref{lem:trivial2}.  Hence
\begin{equation}
  \label{eq:64}
  \deg g' = \mult_{a \in g^{-1}(c)} \mult_{a}(g') = i(c,g) = \deg g -
  \# g^{-1}(c)
\end{equation}
by \eqref{eq:sumdeg}.  Interchanging the r\^oles of $g$ and $g^{*}$ shows $\deg {g^{*}}' = \deg g^{*} - \#
{g^{*}}^{-1}(c)$ and \autoref{cor:ass}~\ref{cor:ass4} yields $\deg g' =
\deg {g^{*}}'$, thus $\# g^{-1}(c) = \# {g^{*}}^{-1}(c)$.
\end{proof}

The previous lemmas deal with the root multiplicities over $\kk$. The next lemma shows that certain parameters are in $\ff$, when $\ff$ is assumed to be perfect.
\begin{lemma}
\label{lem:rational}
Let $\ff$ be perfect, $c \in \kk$, $r$ be a power of $p$, and $f \in P_{r^2}(\ff)$ have a $2$-collision $\{ (g,h),(g^*, h^*) \}$ satisfying \autoref{assum:a}. Then the following hold.
\begin{ronumerate}
\item\label{item:rat1} If  $f-c$ is squareful, then $c \in \ff$.
\item\label{item:rat2} If $g-c = g_1^{m_1} g_2^{m_2}$ for some monic squarefree coprime polynomials $g_1, g_2 \in \kk[x]$ and integers $m_1 \neq m_2$, then $c \in F$ and $g_1, g_2 \in \ff[x]$.
\item\label{item:rat3} If $a \in \ff$ and $h-a = h_1^{m_1} h_2^{m_2}$
  for some monic squarefree coprime polynomials $h_1, h_2 \in \kk[x]$
  and positive integers $m_1 \neq m_2$, then $h_1, h_2 \in \ff[x]$.
\end{ronumerate}
\end{lemma}
\begin{proof}
 Since $\ff$ is perfect, $\kk$ is Galois over $\ff$. An element $c \in \kk$ is fixed by all automorphisms in the Galois group $\gal (\kk \mid \ff)$ if and only if $c \in \ff$.

 \ref{item:rat1} Let $f-c$ be squareful and $\sigma \in \gal (\kk \mid \ff)$. Then $\sigma(f-c) = f - \sigma(c)$ is squareful as well. Indeed, if $f-c = (x-a)^2 u$ for some $a\in\kk$ and $u\in \kk[x]$, then $\sigma(f-c) =(x - \sigma a)^{2} \sigma(u)$. But by \autoref{lem:oneram}~\ref{item:oneram3}, $c$ is unique and thus $c = \sigma(c)$. This holds for all $\sigma \in \gal (\kk \mid \ff)$ and hence $c\in\ff$.

\ref{item:rat2} Since $m_1 \neq m_2$, $g-c$ is squareful and thus $f-c$ is squareful, by \autoref{lem:squareful}. By \ref{item:rat1}, we find $c\in \ff$. Let $\sigma \in \gal (\kk \mid \ff)$. Then $g_1^{m_1} g_2^{m_2} = g-c = \sigma(g-c) = \sigma(g_1)^{m_1} \sigma(g_2)^{m_2}$. Since $m_1 \neq m_2$, unique factorization implies that $g_i = \sigma(g_i)$ and thus $g_i \in \ff[x]$ for $i = 1,2$.

The proof of \ref{item:rat3} is analogous to that of \ref{item:rat2}.
\end{proof}

\section{Classification}
\label{sec:classification}

We use the results of the previous section to describe in
\autoref{thm:class2} the factorization of the components of 2-collisions at degree $r^{2}$ satisfying \autoref{assum:a} over a perfect field $\ff$.  All non-Frobenius collisions at degree $p^{2}$ satisfy this assumption and in \autoref{thm:normal} we provide a complete classification of 2-collisions at that degree over a perfect field. That is, the 2-collisions at degree $p^2$ are up to original shifting those of \autoref{exa:frob}, \autoref{thm:nonadd}, and \autoref{thm:constmulti}.  This yields the maximality of these collisions (\autoref{cor:constructions-are-maximal}) and an efficient algorithm to determine whether a given polynomial $f \in P_{p^{2}}(\ff)$ has a 2-collision (\autoref{algo:coll-det}).  In the next section we use this classification to count exactly the decomposable polynomials over a finite $\ff$.

Let $F$ be a perfect field and denote by $\kk = \overline{F}$ an algebraic closure of $\ff$.
\begin{definition}
\label{def:moso}  Let $r$ be a power of $p$ and $f \in P_{r^2}(\ff)$ have a $2$-collision $\{ (g,h),(g^*, h^*) \}$ satisfying \autoref{assum:a}.  We call $f$ \emph{multiply original} if there is some $c\in \kk$ such that $f-c$ has no simple roots in $K$.  Otherwise, we call $f$ \emph{simply original}.
\end{definition}

By \autoref{lem:oneram}~\ref{item:oneram3}, there is at most one $c\in\kk$ such that $f-c$ is squareful.
Since $F$ is perfect, such a $c$ lies in $\ff$ if it exists, by \autoref{lem:rational}~\ref{item:rat1}.
 Furthermore, if $f$ is multiply original, then there is some $c \in F$ such that $f-c$ is squareful. If $f$ is simply original, then either $f-c$ is squarefree for all $c\in\kk$ or there is a unique $c\in\ff$ such that $f-c$ is squareful and has a simple root.

\begin{example}
\autoref{assum:a} holds for the 2-collisions $\MultConst{a,b,m}$ of \autoref{thm:constmulti} and the $\#T$-collisions $\SimpConst{u, s, \varepsilon, m}$ of \autoref{thm:nonadd} with $\#T \geq 2$; see \autoref{ex:assumconst}.  Moreover, a polynomial $\MultConst{a,b,m}$ has no simple roots and is therefore multiply original.  When $\# T \geq 2$, then $f = \SimpConst{u, s, \varepsilon, m}$ is squareful with a simple root if $m>1$, and $f-c$ is squarefree for all $c \in \kk$ if $m=1$.
\end{example}

\autoref{thm:class2} and \autoref{thm:normal} answer the converse question, namely whether every simply original or multiply original polynomial can be obtained as $\SimpConst{u, s, \varepsilon, m}^{[w]}$ or $\MultConst{a,b,m}^{[w]}$, respectively.  We need the following graph-theoretic lemma.
\begin{lemma}
\label{lem:graph}
Let $G=(V ,E)$ be a directed bipartite graph with bipartition $V = A
\cup A^{*}$, where the outdegree of each vertex equals $\ell > 1$ and $\# A = \# A^{*} = \ell+1$. Then some vertex in $A$ is connected to all other vertices in $A$ by a path of length~2.
\end{lemma}
\begin{proof}
Let $A=\{0,\dots,\ell\}$, $A^{*} = \{\ell+1, \dots, 2\ell+1\}$, and
$M$ the $(2\ell+2)\times(2\ell+2)$ adjacency matrix of $G$ having for each edge from $i \in A
\cup A^{*}$ to
$j \in A \cup A^{*}$ the entry 1 at position $(i,j)$ and entries 0
everywhere else.  Since $G$ is bipartite, we have
\begin{equation}
  \label{eq:20}
  M = \begin{pmatrix}
0 & N \\
N^{*} & 0
\end{pmatrix},
\end{equation}
where $N$ and $N^{*}$ are $(\ell+1)\times(\ell+1)$-matrices satisfying
the following properties by the assumptions of the lemma.
\begin{ronumerate}
\item\label{it:i} Each row in $N$ has exactly one entry 0 and all
  other entries 1.
\item\label{it:ii} Exactly $\ell + 1$ entries of $N^{*}$ are 0 and all
  other entries are 1.
\end{ronumerate}
The number of paths of length 2 that connect a vertex $i \in A$ to a
vertex $j \in A$ is given by the entry $(i,j)$ of
\begin{equation}
  \label{eq:25}
  M^{2} = \begin{pmatrix}
N\cdot N^{*} & 0 \\
0 & N^{*}\cdot N
\end{pmatrix}.
\end{equation}
If every column of $N^{*}$ contains at least two 1's, then $N \cdot
N^{*}$ has only positive entries, because of \ref{it:i}, and every
vertex in $A$ is connected to all other vertices in $A$ by a path of
length 2.  Otherwise, $N^{*}$ has a column $j$ that contains at most one 1.
Because of \ref{it:ii}, every different column of $N^{*}$ contains at
most one 0.  Because of $\ell > 1$ and \ref{it:i}, all entries at
$(j,j')$ with $j' \neq j$ in $N \cdot N^{*}$ are positive.  Starting
from $j$ we can reach all other vertices by a path of length 2.
\end{proof}
Thanks go to Rolf Klein and an anonymous referee for this proof, much simpler than our original one.

\begin{proposition}
\label{thm:class2}
Let $r$ be a power of the characteristic $p > 0$ of the perfect field
$\ff$ and let $f \in P_{r^2}(\ff)$ have a $2$-collision $\{ (g,h),(g^*, h^*) \}$ satisfying \autoref{assum:a}.  Then exactly one of the following holds.
\begin{itemize}
\item[\textnormal{(s)}] \namedlabel{s}{(s)} The polynomial $f$ is simply original.
Let $m = (r-1)/(r-1-\deg g')$. Then there are $w\in \ff$ and unique monic squarefree polynomials $\hat{f}$, $\hat{g}$, $\hat{h}$, $\hat{g}^{*}$, and $\hat{h}^{*}$ in $\ff[x]$, none of them divisible by $x$, with $\hat{f}$ of degree $(r^2-1)/m$ and the other four polynomials of degree $r-1-\deg g' = (r-1)/m$ such that
\begin{equation}
    \label{eq:71}
  \begin{split}
f^{[w]} & = x\hat{f}^{m}, \\
g^{[h(w)]} & = x \hat{g}^m, \\
h^{[w]} & = x\hat{h}^m, \\
(g^{*})^{[h^{*}(w)]} & = x (\hat{g}^{*})^m, \\
(h^{*})^{[w]} & = x(\hat{h}^{*})^m.
  \end{split}
\end{equation}
If $\deg f' > 0$, then $w$ is unique.  Otherwise, factorizations \eqref{eq:71} with the claimed properties exist for all $w \in \ff$.
\item[\textnormal{(m)}] \namedlabel{m}{(m)}
The polynomial $f$ is multiply original and there are  $a$,  $b$, and $m$ as in \autoref{thm:constmulti} and $w \in \ff$ such that
  \begin{equation}
  f^{[w]} = \MultConst{a,b,m}
  \end{equation}
and the collision $\{(g,h)^{[w]}, (g^*,h^*)^{[w]}\}$ is as in \autoref{thm:constmulti}.
\end{itemize}
\end{proposition}

\begin{proof}
Every polynomial satisfying the assumption of the proposition is either simply original or multiply original by \autoref{def:moso}.  So, at most one of the two statements holds and it remains to exhibit the claimed parameters in each case. We begin with two general observations.
\begin{ronumerate}
\item\label{item:3} If $\deg f'=0$,
then $f-c$ is squarefree for all $c \in \kk$, by \autoref{lem:trivial2}. Thus $f$ is simply original. Moreover, $f^{[w]}$ has derivative $(f^{[w]})' = f' \circ (x+w) = f' \in \ff^{\times}$ for all $w \in \ff$ and is therefore squarefree.
\item\label{item:7} If $\deg f'>0$, then there is some $c \in \kk$ such that $f-c$ has a multiple root by \autoref{lem:trivial2}. Moreover, $c$ is unique by \autoref{lem:oneram}~\ref{item:oneram3}, and in $\ff$ by \autoref{lem:rational}~\ref{item:rat1}.  Let $\# \Van{g}{c} = \ell + 1$ be the number of distinct roots of $g-c$ in $\kk$.  By \autoref{lem:oneram}~\ref{item:samenumber}, $g^{*}-c$ also has $\ell + 1$ roots in $\kk$ and
\begin{equation}
  \label{eq:78}
\ell = r - 1 - \deg g' \geq 1,
\end{equation}
by \eqref{eq:64}.  Let $a_{0}, \dots, a_{\ell}$ and $a_{0}^{*}, \dots, a_{\ell}^{*}$ be the distinct roots of $g -c $ and $g^* - c$, respectively, and let  $e_i = \mult_{a_i}(g - c)$ and $\et_i = \mult_{a_i^*} (g^* - c)$ be their multiplicities, that is,
\begin{equation}
\label{eq:66}
  g -c  = \prod_{0 \leq i \leq \ell} (x-a_{i})^{e_{i}}, \quad g^{*} -c  = \prod_{0 \leq i \leq \ell} (x-a^{*}_{i})^{e^{*}_{i}}.
\end{equation}
By \autoref{prop:rammulti}, for each $i$ and $j$ the set $B_{i,j} = \Van{h}{a_{i}} \cap \Van{{h^{*}}}{a_{j}^{*}} \subseteq K$ has cardinality $\gcd(e_{i}, e_{j}^{*})$.
\end{ronumerate}
We now deal with the two cases of the theorem separately.

Case \ref{s}: Let $f$ be simply original.  First, if $\deg f'=0$, then $\deg g' = 0$, by \eqref{eq:derivdeg}. Hence $m=(r-1)/(r-1-\deg g')=1$ and $f^{[w]} = g^{[h(w)]} \circ h^{[w]} = (g^{*})^{[h^{*}(w)]} \circ (h^{*})^{[w]}$ is squarefree for all $w \in \ff$, by \ref{item:3}.  Thus the monic polynomials $\hat{f}=f^{[w]}/x$, $\hat{g} = g^{[h(w)]}/x$, $\hat{h} = h^{[w]}/x$, $\hat{g}^{*} = (g^{*})^{[h^{*}(w)]}/x$, and $\hat{h}^{*} = (h^{*})^{[w]}/x$ are also squarefree and not divisible by $x$, and \eqref{eq:71} holds for all $w \in \ff$.

Second, we assume $\deg f'>0$ for the rest of case \ref{s}.  By \ref{item:7}, there is a unique $c\in \ff$ such that $f-c$ has multiple roots and we assume the notation of \eqref{eq:66} for $g-c$ and $g^{*}-c$.  By the definition of simple originality, $f-c$ has a simple root, say $b_{0} \in \Van{f}{c}$. Furthermore, $g-c$ and $g^{*}-c$ also have simple roots, since
\begin{equation}
  \label{eq:65}
1 = \mult_{b_{0}} (f - c) = \lcm ( \mult_{h(b_{0})} (g - c), \mult_{h^*(b_{0})}(g^* - c))
\end{equation}
by \autoref{prop:rammulti}.  But $g-c$ and $g^{*}-c$ have at most one simple root by \autoref{lem:squareful}.  We may number the roots so that these unique simple roots are $a_{0} = h(b_{0})$ and $a_{0}^{*} = h^{*}(b_{0} )$, both with multiplicity $e_0 = \et_0 = 1$, and $e_{i}, e^{*}_{i} > 1$ for all $i \geq 1$.

By \autoref{lem:c}~\ref{item:c2} and using $e_{0}^{*} = 1$, each
$e_{i}$ with $i \geq 1$ divides all $e^{*}_{j}$ with $j \geq 1$.
Similarly, each $e_{j}^{*}$ with $j \geq 1$ divides all $e_{i}$ with
$i \geq 1$.  Thus all these multiplicities are equal to some integer
$m \geq 2$, and with $r = \deg g = 1 + \ell m$ from \eqref{eq:66}, we have $m=(r-1)/\ell
= (r-1)/(r-1-\deg g')$ by \eqref{eq:78}.  Therefore
\begin{equation}
  \label{eq:67}
g - c = (x - a_0) \tilde{g}^m, \quad g^* - c = (x-a^*_0) ({\tilde{g}}^{*})^{m}
\end{equation}
with monic squarefree polynomials $\tilde{g} = \prod_{1 \leq i \leq \ell} (x -a_{i})$ and $\tilde{g}^* = \prod_{1 \leq i \leq \ell} (x -a^*_{i}) \in \kk[x]$.  We find $a_0, a_0^* \in \ff$ and $\tilde{g}$, $\tilde{g}^* \in \ff[x]$ by \autoref{lem:rational}~\ref{item:rat2}.

Next, we show that $h-a_{0}$ and $h^{*} - a_{0}^{*}$ have the same
root multiplicities as $g^{*}-c$ and $g-c$, respectively.  For $0 \leq
i \leq \ell$, we find from \eqref{eq:5070} with the unique $b_{i} \in
B_{0,i}$ and the unique $b_{i}^* \in B_{i,0}$ as implicitly defined in \ref{item:7} that
\begin{align}
  \mult_{b_{i}}(h-a_{0}) & = \lcm(\mult_{a_{0}} (g - c), \mult_{a_{i}^{*}}(g^* - c) ) = \mult_{a_{i}^{*}}(g^* - c), \\
  \mult_{b_{i}^*} (h^{*} - a_{0}^{*}) & = \mult_{a_{i}} (g-c).
\end{align}
Since $\# B_{0,0} = 1$ by \autoref{prop:rammulti}, we have $b_{0}  = b_{0}^{*}$ and arrive at
\begin{equation}
  h - a_{0} = (x-b_{0}) \tilde{h}^{m}, \quad h^* - a^*_{0} = (x-b_{0}) (\tilde{h}^*)^{m}
\end{equation}
with monic squarefree polynomials $\tilde{h} = \prod_{1 \leq i \leq \ell} (x - b_{i})$ and $\tilde{h}^* = \prod_{1 \leq i \leq \ell} (x - b^*_{i})  \in \kk[x]$. Again, we find $b_{0} \in \ff$ and $\tilde{h}$, $\tilde{h}^* \in \ff[x]$, by \autoref{lem:rational}~\ref{item:rat3}.

Finally, we let $w=b_{0}$, $\hat{g} = \tilde{g} \circ (x + a_0)$, $\hat{h} = \tilde{h} \circ (x + b_{0})$, $\hat{g}^{*} = \tilde{g}^{*} \circ (x+a_{0}^{*})$, $\hat{h}^{*} = \tilde{h}^{*} \circ (x+b_{0})$, and $\hat{f} = \hat{h} \cdot \hat{g}(x\hat{h}^{m})$.  Then $h(b_0) = a_0$, $f(b_{0}) = g(h(b_{0})) = g(a_0) = c$, and
\begin{align}
g^{[h(w)]} & = (x - c) \circ g \circ (x + a_0) = x \hat{g}^m, \\
h^{[w]} & = (x - a_0) \circ h \circ (x + b_{0}) = x\hat{h}^m, \\
(g^{*})^{[h^{*}(w)]} & = (x - c) \circ g^* \circ (x + a^*_0) = x (\hat{g}^{*})^m, \\
(h^{*})^{[w]} & = (x - a^*_0) \circ h^* \circ (x + b_{0}) = x(\hat{h}^{*})^m
\end{align}
with squarefree monic $\hat{g}$, $\hat{g}^*$, $\hat{h}$, and $\hat{h}^*$ of degree $\ell = r - 1 - \deg g'$.  Furthermore, $\hat{g}(0) = \tilde{g}(a_{0}) = \prod_{1 \leq i \leq \ell} (a_{0}-a_{i}) \neq 0$.  This shows that $\hat{g}$ is coprime to $x$ and similar arguments work for $\hat{g}^{*}$, and for $\hat{h}$ and $\hat{h}^{*}$ with $b_{0} \neq b_{i}$ for $i \geq 1$, since $h(b_{0}) = a_{0} \neq a_{i} = h(b_{i})$ for $i\geq 1$.  Moreover, $\hat{f} = \hat{h} \cdot \prod_{1 \leq i \leq \ell} (x\hat{h}^{m} - a_{i} + a_0)$ is monic and not divisible by $x$, and $f^{[w]}  = g^{[h(w)]} \circ h^{[w]} = (x \hat{g}^{m}) \circ (x\hat{h}^{m}) = x \hat{f}^{m}$.
Since $B_{0,0} = \{b_0\}$ and $\lcm (e_0, e^*_0) = 1$, we find that $f-c$ has a simple root $b_0$, by \autoref{prop:rammulti}. Furthermore, $f-c$ has $\sum_{i+j \geq 1} \# B_{i,j} = 2\ell + \ell^{2} m = \ell(r+1)$ roots with multiplicity $m$. Thus $\hat{f}$ is squarefree and of degree $\ell(r+1) = (r^2 - 1)/m$, and the values as claimed in \ref{s} indeed exist.

For the uniqueness in the case $\deg f' >0$, we consider another factorization
$f^{[w_{0}]} = x \hat{f}_{0}^{m}$ satisfying the conditions of case
\ref{s}.  Then $f(x) - f(w) = f^{[w]} \circ (x-w) = (x-w) (\hat{f}(x-w))^{m}$
and $f(x) - f(w_{0}) = f^{[w_{0}]} \circ (x-w_{0}) = (x-w_{0}) (\hat{f_{0}}(x-w_{0}))^{m}$.  The value for $c$ such that $f-c$ is squareful with a simple root is unique for a simply original polynomial with $\deg f'>0$, as remarked in \ref{item:3}.  Thus $c = f(w) = f(w_{0})$ and $(x-w) (\hat{f}(x-w))^{m} = (x-w_{0}) (\hat{f_{0}}(x-w_{0}))^{m}$.  Since $\deg f' > 0$, we have $\deg g' >0$ and $m > 1$. Unique factorization yields $w = w_{0}$ and $\hat{f} = \hat{f_{0}}$.  An analogous argument works for $\hat{g}$, $\hat{g}^{*}$, $\hat{h}$, and $\hat{h}^{*}$.

This concludes case \ref{s}, and we continue with the case \ref{m}.

Case \ref{m}: Let $f$ be multiply original.  Then $\deg f' > 0$ by \ref{item:3} from the beginning of the proof.  By \ref{item:7}, there is a unique $c \in \ff$ such that $f -c$ is squareful, and then $f-c$ has no simple root by \autoref{def:moso} of multiple originality.  By \autoref{lem:squareful}, $g-c$ and $g^*-c$ are also squareful.

Assume that $g-c$ has a simple root.  Then $\ell > 0$ and we may
number the roots of $g-c$ such that $e_{0} = 1$ in the notation
\eqref{eq:66}.  By \autoref{lem:squareful}, $g-c$ has at most one
simple root and thus $e_{1} > 1$.  By \autoref{lem:c}~\ref{item:c2},
$e_{1}$ divides all $e_{j}^{*}$ but one and we may number the roots of
$g^{*}-c$ such that $e_{1} \mid e_{j}^{*}$ for  $1 \leq j \leq \ell$.
Interchanging the r\^oles of $g$ and $g^*$ in
\autoref{lem:c}~\ref{item:c2}, we have $e_{0}^{*} \mid e_{1}$ since
$e_{0} =1$.  Combining these divisibility conditions shows $e_{0}^{*}
\mid \gcd\{e_{j}^{*} \colon 0 \leq j \leq \ell\}$
and we find $e_{0}^{*} = 1$ from \eqref{eq:49}.
Hence
there exists some $b\in\kk$ such that $\mult_b(f-c) = \lcm(e_0,
e^*_0)=1$, by \autoref{prop:rammulti}, contradicting
\autoref{def:moso} of multiply original by the uniqueness of $c$.
Therefore $g-c$ has no simple root and $e_{i}>1$ for all $i \geq 0$.
An analogous argument for $g^{*}$ shows $e_{i}^{*}>1$ for all $i \geq 0$.

We now proceed in three steps.  First, we determine the factorizations
of $g-c$ and $g^{*}-c$.  Second, we derive the factorizations of
$h-a_{i}$ and $h^{*}-a_{i}^{*}$ for the roots $a_{i} \in \Van{g}{c}$ and
$a_{i}^{*} \in \Van{{g^{*}}}{c}$, respectively.  Third, we apply an appropriate original shift and prove the claimed form.

To compute $\ell$, we translate \autoref{prop:rammulti} into the language of graphs. We consider the directed bipartite graph on the set $V = A \cup A^{*}$ of vertices, with disjoint $A = \{i \colon 0\leq i\leq \ell\}$ and $A^* = \{i^* \colon 0\leq i\leq \ell\}$. The set $E$ of edges consists of all $(i, j^*)$ with $e_i \mid \et_j$ plus all $(i^*, j)$ with $\et_i \mid e_j$. Each vertex has outdegree $\ell$, by \autoref{lem:c} \ref{item:c2}, since no root is simple.  If $\ell > 1$, then by \autoref{lem:graph} some vertex $i$ in $A$ is connected to all other vertices in $A$.
Then $e_i > 1$ divides all other multiplicities of $g - c$, which contradicts \eqref{eq:49} with $g$ instead of $g^*$.
Hence $\ell = 1$ and therefore
\begin{equation}
  \label{eq:73}
  \begin{split}
  g - c & = (x-a_{0})^{e_{0}}(x-a_{1})^{e_{1}}, \\
  g^{*}-c & = (x-a_{0}^{*})^{e_{0}^{*}}(x-a_{1}^{*})^{e_{1}^{*}}
\end{split}
\end{equation}
with $1 < e_{i}, e_{i}^{*} < r-1$, for $i= 0,1$.  We know by
\autoref{lem:c}~\ref{item:c1} applied to $g$ and $g^{*}$,
respectively, that $\gcd(e_{0}, e_{1}) =
\gcd(e_{0}^{*}, e_{1}^{*}) = 1$ and since $e_{i},
e_{i}^{*} > 1$ for $i = 0,1$, each $e_{i}$ divides exactly one
$e_{j}^{*}$, by \ref{item:c2} of the cited lemma, and similarly each
$e_{j}^{*}$ divides exactly one $e_{i}$.  By renumbering if required,
we assume $e_{0} \mid e_{1}^{*}$.  If $e_{1}^{*} \mid e_{1}$, then
$\gcd(e_{0}, e_{1}) = e_{0} > 1$, a contradiction to
\autoref{lem:c}~\ref{item:c1}.  Therefore $e_{1}^{*} \mid e_{0}$ and
we have $e_{0} = e_{1}^{*}$.  Similar arguments show $e_{0}^{*} \mid
e_{1}$ and $e_{1} \mid e_{0}^{*}$, and hence $e_{1} = e_{0}^{*}$.  We write $m = e_{0} = e_{1}^{*}$
 and $m^* = e_{1} = e_{0}^{*}$.
Then $m$ and $m^{*}$ are coprime, $m^{*}  =  r - m$, since $r = e_0 + e_1$, and $p \nmid m$, by
\ref{assum:ii}.
\autoref{lem:rational}~\ref{item:rat2} yields distinct $a_{0}, a_{1}
\in \ff$ and distinct $a_{0}^{*}, a_{1}^{*} \in \ff$ with
\begin{equation}
  \label{eq:76}
\begin{split}
  g - c & = (x-a_{0})^{m}(x-a_{1})^{m^{*}}, \\
  g^{*}-c & = (x-a_{0}^{*})^{m^{*}}(x-a_{1}^{*})^{m}.
\end{split}
\end{equation}

For the sets $B_{i,j}$ defined in \ref{item:7}, we find $\# B_{0,0} = \# B_{1,1} =1$, $\#
B_{0,1} = m$, and $\# B_{1,0} = m^{*}$. The multiplicity of each $b_{i,j} \in B_{i,j}$ satisfies
\begin{equation}
  \label{eq:74}
  \mult_{b_{i,j}} (h-a_{i}) = \frac{\lcm(e_{i} , e_{j}^{*})}{e_{i}} = \begin{cases*}
m^{*} & if $i=j=0$, \\
m & if $i = j = 1$, \\
1 & otherwise,
\end{cases*}
\end{equation}
by \eqref{eq:5070}, and similarly
\begin{equation}
  \mult_{b_{i,j}} (h^{*}-a_{j}^{*}) =  \begin{cases*}
m & if $i=j=0$, \\
m^{*} & if $i = j = 1$, \\
1 & otherwise.
\end{cases*}
\end{equation}
Writing $B_{0,0} = \{ b_{0,0} \}$ and $B_{1,1} = \{ b_{1,1} \}$, this shows
  \begin{align}
  h - a_{0} & = (x-b_{0,0})^{m^{*}} H_{0}, & h - a_{1} = (x-b_{1,1})^{m} H_{0}^*, \\
  h^{*} - a_{0}^{*} & = (x-b_{0,0})^{m} H_{0}^*, & h - a_{1}^{*} = (x-b_{1,1})^{m^{*}} H_{0}
  \end{align}
with squarefree monic $H_{0} = \prod_{b\in B_{0,1}} (x-b)$ and $H_{0}^* = \prod_{b \in B_{1,0}} (x-b)$ that do not vanish at $b_{0,0}$ or $b_{1,1}$. \autoref{lem:rational}~\ref{item:rat3} implies that $b_{0,0}, b_{1,1} \in \ff$ and $H_{0}, H_{0}^* \in \ff[x]$.

We use this information to apply the appropriate original shift to our
decompositions.  Let $w = b_{0,0}$, $a=a_{1}-a_{0}$,
$a^{*}=a_{1}^{*}-a_{0}^{*}$, and $b = b_{1,1} - b_{0,0}$, with all
differences being different from 0, and squarefree monic $H = H_{0} \circ (x+w)$ and $H^{*} = H_{0}^* \circ (x+w)$.  Then $h (w) = a_{0}$, $h^{*} (w) = a_{0}^{*}$, $g(a_{0})=g^{*}(a_{0}^{*})=c$, and
\begin{equation}
    \label{eq:77}
  \begin{split}
	     g^{[h(w)]} & = x^{m} (x-a)^{m^{*}}, \\
	     h^{[w]} & = x^{m^{*}} H, \quad h^{[w]} - a  = (x-b)^m H^{*}, \\
    {g^{*}}^{[h^{*}(w)]} & = x^{m^{*}} (x-a^{*})^{m}, \\
    {h^{*}}^{[w]} & = x^{m} H^{*}, \quad {h^{*}}^{[w]} - a^{*}  = (x-b)^{m^{*}} H.
  \end{split}
\end{equation}
Equations \eqref{eq:77} yield a system of linear equations
\begin{align}
  x^{m^*} \qepol - (x-b)^{m} \qepol^{*} & = a, \\
 -(x-b)^{m^*} \qepol + x^{m} \qepol^{*} & = a^{*}
\end{align}
over $\ff(x)$ in $H$ and $H^*$.  We apply Cramer's rule and find
\begin{align}
  H & = (ax^{m} + a^{*}(x-b)^{m})/b^{r}, \\
  H^{*} & = (a^{*}x^{m^{*}} + a(x-b)^{m^{*}})/b^{r},
\end{align}
and $a+a^{*}=b^{r}$, since $H$ is monic.  Therefore, the polynomials
$H$ and $H^*$ are as in \eqref{eq:48} and $f^{[w]} = g^{[h(w)]} \circ h^{[w]} = x^{mm^{*}} (x-b)^{mm^{*}} H^{m} (H^{*})^{m^{*}} = \MultConst{a,b,m}$, as in \autoref{thm:constmulti}.
\end{proof}

For $2$-collisions at degree $p^{2}$, we can refine the classification
of \autoref{thm:class2}.

\begin{theorem}
\label{thm:normal}
Let $F$ be a perfect field of characteristic $p$ and $f \in P_{p^{2}}(\ff)$. Then $f$ has a $2$-collision $\{(g,h), (g^*, h^*)\}$ if and only if exactly one of the following holds.
  \begin{itemize}
    \item[\textnormal{(F)}] \namedlabel{class:0normal}{(F)} The polynomial $f$ is a Frobenius collision as in \autoref{exa:frob}.
    \item[\textnormal{(S)}] \namedlabel{class:1normal}{(S)} The polynomial $f$ is simply original and there are  $u$, $s$,  $\varepsilon$, and  $m$ as in \autoref{thm:nonadd} and $w \in \ff$  such that
\begin{equation}
f^{[w]} = \SimpConst{u,s,\varepsilon,m}
\end{equation}
and the collision $\{(g,h)^{[w]}, (g^{*},h^{*})^{[w]}\}$ is contained in the $\# T$-collision described in \autoref{thm:nonadd}, with $\# T \geq 2$.
  \item[\textnormal{(M)}] \namedlabel{class:3normal}{(M)} The polynomial $f$ is multiply original and there are  $a$,  $b$, and $m$ as in \autoref{thm:constmulti} and $w \in \ff$ such that
  \begin{equation}
  f^{[w]} = \MultConst{a,b,m}
  \end{equation}
and the collision $\{(g,h)^{[w]}, (g^*,h^*)^{[w]}\}$ is as in \autoref{thm:constmulti}.
 \end{itemize}
\end{theorem}

\begin{proof}
By \autoref{lem:cor:frob}~\ref{lem:frob}, $f$ is a Frobenius collision if and only if $f' = 0$.

The rest of the proof deals with the case $f' \neq 0$.
\autoref{assum:a} holds by \autoref{cor:ass}, the assumptions in \autoref{def:moso} are satisfied, and $f$ is either simply original or multiply original.

For a multiply original $f$, \autoref{thm:class2} yields the claimed
parameters directly, and we now show their existence in the simply original case.

We take $w, m, \hat{g}, \hat{h}$  as in \autoref{thm:class2}~\ref{s} and have
\begin{equation}
  \label{eq:69}
  \begin{split}
    g^{[h(w)]} & = x \hat{g}^m, \\
    h^{[w]} & = x\hat{h}^m.
  \end{split}
\end{equation}
We determine the form of $\hat{g}$ and $\hat{h}$. Let $\ell = \deg \hat{g} = (p-1)/m$.
The derivative of $g^{[h(w)]}$ is $\hat{g}^{m-1} (\hat{g} + m x\hat{g}' )$, and its degree equals $\deg g' = p-1-\ell$, by \eqref{eq:78}.
Thus $\deg g' = (m-1)\ell + \deg(\hat{g} + m x\hat{g}') =  \deg g' + \deg(\hat{g} + m x\hat{g}')$ and $\deg(\hat{g} + m x\hat{g}')=0$.
We write $\hat{g} = \sum_{0 \leq i \leq \ell} \hat{g}_{i}x^{i}$
with $\hat{g}_{i} \in \ff$ for all $i \geq 0$.
Then $\hat{g} + m x\hat{g}'  = \sum_{0 \leq i \leq \ell} (1 + m i )
\hat{g}_i x^i$ and we have $\hat{g}_{0} \neq 0$ and $ (1 + m i)
\hat{g}_i = 0$ for all $i\geq 1$.  Since $1 +mi \neq 0$ in
$\ff$ for $ 1 \leq i < \ell$,  it follows that $\hat{g}_i = 0$ for these
values of $i$. Thus we get $\hat{g} = x^\ell - \hat{g}_0$ and
$\hat{g}_{0} \neq 0$. An analogous argument yields $\hat{h} = x^\ell -
\hat{h}_0$ with $\hat{h}_{0} \neq 0$. Therefore, we find
\begin{equation}
\label{eq:fsubaddcomp}
f^{[w]} = x(x^{\ell(p+1)} - (\hat{h}_0^p +\hat{g}_0)x^\ell +\hat{g}_0 \hat{h}_0)^m.
\end{equation}
Let
\begin{equation}
  \label{eq:82}
  (u, s, \varepsilon, t) = \begin{cases*}
(\hat{g}_{0}\hat{h}_{0}, 1, 0, \hat{h}_{0}) & if $\hat{h}_0^p +\hat{g}_0 = 0$, \\
((\hat{h}_0^p + \hat{g}_0)^{p+1}/(\hat{g}_0 \hat{h}_0)^p, \hat{g}_0 \hat{h}_0 / (\hat{h}_0^p + \hat{g}_0), 1, \hat{h}_0/s) & otherwise.
  \end{cases*}
\end{equation}
In both cases, $u$, $s$, and $t$ are in $\ff^\times$ and the equations $t^{p+1} -\varepsilon u t + u =0$, $\hat{h}_0=st$, $\hat{g}_0 = us^p t^{-1}$, and $f^{[w]} = g^{[h(w)]} \circ h^{[w]} = \SimpConst{u, s,\varepsilon, m}$ hold.
Similarly, we find ${g^*}^{[h^*(w)]} = x(x^\ell - \hat{g}^*_0)^m$ and ${h^*}^{[w]} = x(x^\ell - \hat{h}^*_0)^m$ for some $\hat{g}_0^*$, $\hat{h}_0^* \in \ff^\times$, and derive the parameters $u^*$, $s^*$, $\varepsilon^*$, and $t^*$ analogously. Since $f^{[w]} = {g^*}^{[h^*(w)]} \circ {h^*}^{[w]}$, it follows from \eqref{eq:fsubaddcomp} that ${\hat{h}}_0^p +\hat{g}_0 = ({\hat{h}_0}^*)^p +\hat{g}^*_0$ and $\hat{g}_0 \hat{h}_0 = \hat{g}^*_0 \hat{h}^*_0$. Hence $\varepsilon  = \varepsilon^*$, $u = u^*$, and $s = s^*$. Since the decompositions are distinct, we have $t \neq t^*$ and thus $(g,h)^{[w]}$ and $(g^*, h^*)^{[w]}$ are both of the form \eqref{eq:80} with different values for $t$.
\end{proof}

\begin{corollary}
\label{cor:constructions-are-maximal}
\begin{ronumerate}
\item \label{item:max1} A polynomial in case \ref{class:1normal} of \autoref{thm:normal} has a maximal $\# T$-collision with $T$ as in \eqref{eq:7}.
\item \label{item:max2} A polynomial in case \ref{class:3normal} of \autoref{thm:normal} has a maximal $2$-collision.
\end{ronumerate}
\end{corollary}

\begin{proof} For a polynomial $f$ with collision $C$ and $w \in F$, we write
  $C^{[w]} = \{(g,h)^{[w]} \colon (g,h) \in C\}$ for the corresponding
  collision of $f^{[w]}$.

If $f$ is a Frobenius collision as in case
  \ref{class:0normal} of \autoref{thm:normal}, then $f$ is maximal by
  \autoref{lem:cor:frob}~\ref{cor:frob}.  Now let $f$ be a polynomial
  with a $2$-collision $C = \{(g,h),
  (g^{*},h^{*})\}$ that does not fall into case \ref{class:0normal} of
  \autoref{thm:normal}.

\ref{item:max1} If $f$ falls into case \ref{class:1normal} of
  \autoref{thm:normal}, we have by that theorem $u$, $s$, $\varepsilon$, and $m$ as in
  \autoref{thm:nonadd} and $w \in F$ such that
$f = \SimpConst{u,s,\varepsilon,m}^{[-w]}$ and $C \subseteq D(u,s,\varepsilon,m)^{[-w]}$,
where $D(u,s,\varepsilon,m)^{[-w]}$ denotes the $\# T$-collision
described in \autoref{thm:nonadd} shifted by $-w$.

Take  another
decomposition $(g_{0}, h_{0}) \neq (g,h)$ of $f$.  We apply \autoref{thm:normal} to $f$ with
$2$-collision $C_{0} = \{(g,h), (g_{0},h_{0})\}$.  Due to the mutual
exclusivity of the three cases this falls again in case
\ref{class:1normal}, and we obtain
$u_{0}$, $s_{0}$, $\varepsilon_{0}$, and $m_{0}$ as in
  \autoref{thm:nonadd}, and $w_{0}\in F$ such that
$f = \SimpConst{u_{0},s_{0},\varepsilon_{0},m_{0}}^{[-w_{0}]}$
and $C_{0} \subseteq
D(u_{0},s_{0},\varepsilon_{0},m_{0})^{[-w_{0}]}$.  Thus,
\begin{equation}
  \label{eq:23}
f^{[w_{0}]} = \SimpConst{u,s,\varepsilon,m}^{[w_{0}-w]} = \SimpConst{u_{0},s_{0},\varepsilon_{0},m_{0}}.
\end{equation}
By \autoref{lem:unique1}~\ref{item:5}, the only polynomial of the form
\eqref{eq:7} in the orbit of $\SimpConst{u,s,\varepsilon,m}$ under
original shifting is the polynomial itself.  Therefore,
\begin{equation}
  \label{eq:23b}
\SimpConst{u,s,\varepsilon,m} = \SimpConst{u_{0},s_{0},\varepsilon_{0},m_{0}}.
\end{equation}
If $m >1$, then the stabilizer of $\SimpConst{u,s,\varepsilon,m}$
under original shifting is $\{0\}$ by \autoref{lem:unique1}~\ref{item:4} and
we have $w=w_{0}$.  Otherwise, $m=1$ and $\SimpConst{u,s,\varepsilon,m}$,
$D(u,s,\epsilon,m)$, and $D(u_{0},s_{0},\epsilon_{0},m_{0})$ consist
only of additive polynomials which are invariant under
original shifting.  In that case, we can assume $w = w_{0}$ without loss of
generality.

If $\varepsilon = 1$, then
\autoref{lem:unique1}~\ref{item:ii} yields
$(u,s,\varepsilon,m) = (u_{0},s_{0},\varepsilon_{0},m_{0})$ from \eqref{eq:23b}
and therefore
$ D(u,s,\varepsilon,m)^{[-w]} = D(u_{0}, s_{0}, \varepsilon_{0},
m_{0} )^{[-w_{0}]} \ni (g_{0},h_{0})$.
Otherwise, $\varepsilon=0$ and
\autoref{lem:unique1}~\ref{item:iv} yields $(us^{p+1},\varepsilon,m) =
(u_{0}s_{0}^{p+1},\varepsilon_{0},m_{0})$ from \eqref{eq:23b}.  By
 the definition of $D(u_{0},
s_{0}, \varepsilon_{0}, m_{0} )^{[-w_{0}]}$ via \autoref{thm:nonadd}, there is some $t_{0} \in F$
satisfying $t_{0}^{p+1} = -u_{0}$ such that
\begin{align}
  g_{0}^{[h_{0}(-w_{0})]} & = x (x^{p-m_{0}} - u_{0}s_{0}^{p}t_{0}^{-1})^{m_{0}} =
  x   (x^{p-m} - us^{p}t^{-1})^{m}, \\
h_{0}^{[-w_{0}]} & = x(x^{p-m_{0}}-s_{0}t_{0})^{m_{0}} = x(x^{p-m}-st)^{m}
\end{align}
for $t=t_{0}s_{0}/s \in F$.  Since $t$ satisfies $t^{p+1} = -u$, this shows $(g_{0},h_{0}) \in
D(u,s,\varepsilon,m)^{[-w]}$.

\ref{item:max2} Let $f$ fall into case \ref{class:3normal} of
  \autoref{thm:normal} and take another decomposition $(g_{0}, h_{0}) \neq (g,h)$
  of $f$.  We apply that theorem to $f$ with $2$-collisions
  $C$ and $C_{0}= \{(g,h),
  (g_{0},h_{0})\}$ and obtain $a$, $b$, $m$ and
  $a_{0}$, $b_{0}$, $m_{0}$ as in \autoref{thm:constmulti} and $w,
  w_{0} \in F$, respectively, such that
  \begin{gather}
    f  = \MultConst{a,b,m}^{[-w]} =
    \MultConst{a_{0},b_{0},m_{0}}^{[-w_{0}]}, \label{eq:19} \\
 C  = E(a,b,m)^{[-w]}, \quad \text{and} \quad C_{0} = E(a_{0},b_{0},m_{0})^{[-w_{0}]},
  \end{gather}
where $E(a,b,m)^{[-w]}$ denotes the $2$-collision defined in
\eqref{eq:3normal} shifted by $-w$, and
$E(a_{0},b_{0},m_{0})^{[-w_{0}]}$ is analogous.  We have
\begin{equation}
  \label{eq:24}
  \MultConst{a,b,m}^{[w_{0}-w]} = \MultConst{a_{0},b_{0},m_{0}}.
\end{equation}
The only polynomials in the orbit of $\MultConst{a,b,m}$ that are
of the form \eqref{eq:3normal} are $\MultConst{a,b,m}$ itself and
$\MultConst{a,b,m}^{[b]}$ according to
\autoref{pro:multi_uniqueness}~\ref{item:2}; and by \ref{item:1}, the stabilizer of
$\MultConst{a,b,m}$ under original shifting is $\{0\}$.  Hence, $w_{0}-w=0$ or $w_{0}-w=b$.

If $w_{0} = w$, then $\MultConst{a_{0},b_{0},m_{0}}  = \MultConst{a,b,m}$ from \eqref{eq:24} and with \ref{item:0} of the cited proposition
\begin{equation}
  \label{eq:22}
(a_{0},b_{0},m_{0}) \in \{(a,b,m), (a^{*},b,m^{*})\}.
\end{equation}
If $w_{0} = w + b$, then $\MultConst{a_{0},b_{0},m_{0}} = \MultConst{a,b,m}^{[b]} =
\MultConst{-a^{*},-b,m}$ and again with \ref{item:0}
\begin{equation}
  \label{eq:21}
(a_{0}, b_{0}, m_{0}) \in \{(-a^{*},-b,m), (-a,-b,m^{*})\}.
\end{equation}
In either case, we check directly that $E(a_{0},b_{0},m_{0})^{[-w_{0}]} = E(a,b,m)^{[-w]}$
and therefore $(g_{0},h_{0}) \in C$.
\end{proof}

In particular, the polynomials of case \ref{class:3normal} have no
$3$-collision.  We combine \autoref{thm:normal} with the algorithms of
\autoref{sec:explicit-construction} for a general test of
$2$-collisions in \autoref{algo:coll-det}.

\begin{algorithm2f}
\caption{Collision determination}
\label{algo:coll-det}
\KwIn{a polynomial $f \in P_{p^{2}}(F)$, where $p= \chara F$}
\KwOut{``\ref{class:0normal}'', ``\ref{class:1normal}'', or ``\ref{class:3normal}'' as in \autoref{thm:normal}, if $f$ has a $2$-collision, and ``no 2-collision'' otherwise}

\lIf{$f \in F[x^p] \setminus \{x^{p^{2}}\}$}{\KwRet{``\ref{class:0normal}''}}
\If{\autoref{algo:recover-simple-paras} does not return ``failure'' on input $f$, but $k, u, s, \varepsilon, m, w$}{
  \lIf{$k \geq 2$}{\KwRet{``\ref{class:1normal}''}}
}
\If{\autoref{algo:recover-multi-paras-infinite-fields} does not return ``failure'' on input $f$}{
  \KwRet{``\ref{class:3normal}''}\;
}
\KwRet{``no 2-collision''}
\end{algorithm2f}

\begin{theorem}
 \autoref{algo:coll-det}
 works correctly as specified.  If $F = \Fq$ and $n = p^{2} = \deg f$, it takes
$O(\MM(n) \log(pq))$
 field operations.
\end{theorem}

The correctness follows from \autoref{thm:normal}.  Its cost is dominated by that of \autoref{algo:recover-simple-paras}, where the $\log n$ factor is subsumed in $\log(pq)$ since $n=p^2$ and $pq \geq p^2$.  If $f$ is found to have a collision, then that can be returned as well, using \autoref{exa:frob} for \ref{class:0normal}.

\section{Counting at degree $p^{2}$}
\label{sec:Counting}

The classification of the composition collisions at degree $p^2$ yields the exact number of decomposable polynomials over a finite field $\Fq$.

\begin{theorem}
\label{thm:count_coll}
Let $p$ be a prime and $q$ a power of $p$. For $k \geq 1$, we write $c_{k}$ for $\# C_{p^{2},k} (\Fq)$ as in \eqref{eq:47}, $\delta$ for Kronecker's delta function, and $\divfct$ for the number of positive divisors of $p-1$.  Then the following hold.
\begin{align}
  c_{1} & = q^{2p-2} - 2q^{p-1} + 2 - \frac{(\divfct q-q+1)(q-1)(qp-q-p)}{p} \\
& \quad  - (1 - \delta_{p,2}) \frac{q(q-1)(q-2)(p-3)}{2}, \label{eq:37} \\
  c_{2} & =  q^{p-1} - 1 + \frac{(\divfct q - q +1)(q-1)^{2}(p-2)}{2(p-1)} \\
          & \quad + (1 - \delta_{p, 2})\frac{q(q-1)(q-2)(p-3)}{4}, \label{eq:38} \\
  c_{p+1} & = \frac{(\divfct q - q +1)(q-1)(q-p)}{p(p^{2}-1)},  \label{eq:39} \\
  c_{k} & = 0, \quad \text{if $k \notin \{1, 2, p+1\}$}. \label{eq:36}
\end{align}
\end{theorem}

\begin{proof}
For $k \geq 2$, we consider $C_k = C_{p^{2}, k} (\Fq)$.  \autoref{thm:normal} provides the partition
\begin{equation}
  \label{eq:40}
  C_{k} = \Cfrob{k} \,\dcup\, \Csimp{k} \,\dcup\, \Cmult{k},
\end{equation}
where the sets on the right-hand side correspond to the cases \ref{class:0normal}, \ref{class:1normal}, and \ref{class:3normal}, respectively.
\autoref{lem:cor:frob}~\ref{cor:frob}, \autoref{fac:simply_count}, and \autoref{cor:multi_count} imply that

\begin{equation}
  \label{eq:35}
  \# \Cfrob {k} = \begin{cases*}
q^{p-1} - 1 & if $k=2$, \\
0 &  if $k \geq 3$,
  \end{cases*}
\end{equation}
\begin{equation}
  \label{eq:41}
\#  \Csimp{k} = \begin{cases*}
\dfrac{(\divfct q - q +1)(q-1)^{2}(p-2)}{2(p-1)} & if $k = 2$, \\
\dfrac{(\divfct q - q +1)(q-1)(q-p)}{p(p^{2}-1)} & if $k = p+1$, \\
0 & otherwise,
\end{cases*}
\end{equation}
\begin{equation}
\# \Cmult{k} =
\begin{cases*}
  (1 - \delta_{p, 2}) \dfrac{q(q-1)(q-2)(p-3)}{4} & if $k=2$, \\
  0 & if $k \geq 3$.
\end{cases*}
\end{equation}

Summing up yields the exact formulas \eqref{eq:38}, \eqref{eq:39}, and \eqref{eq:36}. Finally, there is a total of $q^{2p-2}$ pairs $(g,h) \in P_{p}(\Fq) \times P_{p}(\Fq)$ and therefore \eqref{eq:37} follows from
\begin{equation}
  c_{1} = q^{2p-2} - \sum_{k\geq 2} k\cdot c_{k}. \tag*{\qedhere}
\end{equation}
\end{proof}

Equation \eqref{eq:missing1} now yields the counting result of this paper, namely the following exact formula for the number of decomposable polynomials of degree $p^{2}$ over $\Fq$.

\begin{theorem}
\label{cor:main}
Let $\Fq$ be a finite field of characteristic $p$, $\delta$ Kronecker's delta function, and $\divfct$ the number of positive divisors of $p-1$. Then
  \begin{equation}
    \label{eq:1}
    \begin{split}
    \# D_{p^{2}} (\Fq) & = q^{2p-2} -q^{p-1}  + 1  -  \frac{(\divfct q  -q +1)(q-1)(qp-p-2)}{2(p+1)} \\
    & \quad - (1- \delta_{p, 2}) \frac{q(q-1)(q-2)(p-3)}{4}.
    \end{split}
  \end{equation}
\end{theorem}

\begin{proof}
By \eqref{eq:missing1} and \autoref{thm:count_coll} we find
\begin{equation}
  \# D_{p^{2}} (\Fq) = q^{2p-2} - c_{2} -pc_{p+1},
\end{equation}
from which the claim follows.
\end{proof}

For $p=2$, this yields
\begin{equation}
\# D_{4} (\Fq) = q^{2} \cdot \frac{2+q^{-2}}{3},
\end{equation}
consistent with the result in \cite{gat12}.  Furthermore, we have
\begin{align}
\# D_{9} (\Fq) & = q^{4} \left( 1 - \frac{3}{8} ( q^{-1} + q^{-2} - q^{-3} - q^{-4})\right) && \text{for $p = 3$},\\
\# D_{p^{2}} (\Fq) & = q^{2p-2} \left( 1 - q^{-p+1} + O(q^{-2p+5+1/d}) \right) &&\text{for $q = p^{d}$ and $p \geq 5$}.
\end{align}

We have two independent parameters $p$ and $d$, and $q= p^d$. For two eventually positive functions $f, g \colon \mathbb{N}^{2} \rightarrow \mathbb{R}$, here $g \in O(f)$ means that there are constants $b$ and $c$ so that $g(p,d) \leq c \cdot f(p,d)$ for all $p$ and $d$ with $p+d \geq b$.
With the bounds on $\divfct$ mentioned after the proof of
\autoref{fac:simply_count}, we have the following asymptotics.
\begin{corollary}  Let $p \geq 5$, $d \geq 1$, and $q= p^d$. Then
\begin{align}
  c_{1} & = q^{2p-2} (1 - 2q^{-p+1} + O(q^{-2p+5+1/d})), \\
  c_{2} & = q^{p-1} (1 + O(q^{-p+4+1/d})), \\
  c_{p+1} & = (\tau - 1) q^{3-3/d} \left(1 + O(q^{-\max\{2/d,1-1/d\}})\right) \\
& = O\left(q^{3-3/d+1/(d \loglog p)} \right).
\end{align}
\end{corollary}

\Citet{gat09b} %
considers the asymptotics of
\begin{equation}
  \label{eq:79}
  \nu_{q,n} = \begin{cases*}
\# D_{n}/q^{2\ell-2} & if $n = \ell^{2}$, \\
\# D_{n}/2q^{\ell + n/\ell -2} & otherwise, \\
\end{cases*}
\end{equation}
where $\ell$ is the smallest prime divisor of $n$.  It turns out that for any composite $n$, $\lim\sup_{{q\to\infty}} \,\nu_{q,n} = 1$, and that $\lim\inf_{{q\to\infty}} \,\nu_{q,n} = 1$
for many $n$. But when $\ell$ divides $n$ exactly twice, denoted as $\ell^2 \parallel n$,
determining the limes inferior was left as an open question.
If $n = \ell^{2}$, we obtain from \autoref{cor:main}
\begin{equation}
  \label{eq:28}
  \lim_{q\to\infty} \nu_{q,\ell^2} = 1
\end{equation}
for any prime $\ell > 2$.  For $n=4$, the sequence has no limit, but
oscillates between close to $\lim\inf_{q\to\infty} \,\nu_{q,4} = 2/3$ and $\lim\sup_{q\to\infty} \,\nu_{q,4} = 1$, and these are the only two accumulation points of the sequence $\nu_{q,4}$.
If $\ell^{2} \parallel n$ and $n \neq \ell^{2}$, the question of good asymptotics is still open, as it is for $\nu_{q,n}$ when $q$ is fixed and $n \to\infty$.

\section{Conclusion}
\label{sec:conclusion}

In the wild case of univariate polynomial decomposition, we present some (equal-degree) collisions in the special case where the degree is $r^2$ for a power $r$ of the characteristic $p$, and determine their number.  We give a classification of all 2-collisions at degree $p^2$ and an algorithm which determines whether a given polynomial has a 2-collision, and if so, into which class it falls. We compute the exact number of decomposable polynomials of degree $p^2$ over finite fields.
This yields tight asymptotics on $\nu_{q,n} = (\text{number of
  decomposables of degree }n)/q^{2\ell -2}$ for $q \to \infty$, when $n = \ell^2$
is the square of a prime $\ell$.

Ritt's Second Theorem covers distinct-degree collisions, even in the
wild case, see \cite{zan93}, and they can be counted exactly in most
situations; see \cite{gat10}. It would be interesting to see a similar
classification for general equal-degree collisions.

This paper only deals with decomposition of univariate polynomials.
The study of rational functions with our method remains open.

\section{Acknowledgments}
Many thanks go to Mike Zieve for useful comments and pointers to the
literature and to Rolf Klein and an anonymous referee for simplifying our proof of \autoref{lem:graph}.

This work was funded by the B-IT Foundation and the Land Nordrhein-Westfalen.

\bibliography{journals,refs,lncs}

\bibliographystyle{cc2e}

\listoftodos

\end{document}